\documentclass[12pt]{article}
 
\usepackage{amstext,amssymb,amsmath,stmaryrd,amsthm}
\usepackage{geometry} 
\usepackage{color}
\usepackage{hyperref}
\usepackage{makeidx}
\makeindex

\numberwithin{equation}{section}

\newtheorem{theorem}{Theorem}[section]
\newtheorem{corollary}[theorem]{Corollary}
\newtheorem{lemma}[theorem]{Lemma}
\newtheorem{proposition}[theorem]{Proposition}

\newtheorem{Remark}[theorem]{Remark}

\begin{document}

\title{The Logarithmic Sobolev Inequality for Gibbs measures on  infinite product of Heisenberg groups.}
\author{Ioannis Papageorgiou\thanks{Department of Mathematics, Uppsala University, P.O Box 480, Uppsala 751 06, Sweden.  Email: ioannis.papageorgiou@math.uu.se,  papyannis@yahoo.com}}

\date{}

\maketitle

\begin{abstract}
 We are interested in the $q$ Logarithmic Sobolev inequality for probability measures on the infinite product of Heisenberg groups.  We assume that the one site boundary free measure satisfies either a $q$ Log-Sobolev inequality or a U-Bound inequality, and we determine conditions so that the infinite dimensional Gibbs measure satisfies a $q$ Log-Sobolev inequality.  \end{abstract}


\section{Introduction}
We focus on the $q$ Logarithmic Sobolev Inequality  (LSq) for the infinite dimensional  Gibbs measure related to systems of  spins with values in the Heisenberg group.  More specifically, we extend the already
 known results for real valued spins with  interactions $V$ that satisfy $\left\Vert \nabla_i \nabla_j V(x_i,x_j) \right\Vert_{\infty}\leq\infty$     to the case of  the Heisenberg group.
We investigate two cases, that of a boundary free  one site measure that  satisfies an (LSq) inequality and that of a one site measure with interactions that satisfies a non uniform U-Bound inequality. In both cases we determine conditions so that the infinite dimensional Gibbs measure satisfies an (LSq) inequality.

In this section we present the main definitions as well as some of the most relevant past results concerning the Log-Sobolev inequality on the Heisenberg group and the infinite dimensional setting. 
\subsection{(LSq) on the Heisenberg Group.}
The Heisenberg group is one of the simplest sub-Riemannian settings in which we can define non-elliptic H\"ormander type generators.  However, it has the disadvantage that   when we consider coercive inequalities the fundamental methods for proving the Log-Sobolev inequality for the one site measure  $\mathbb{E}^{\{i\},\omega}$, mainly the  $\Gamma_{2}$ criterion, developed in \cite{B} (see also \cite{B-E})  cannot be applied easily in this case   (see \cite{B-B-B-C}).

~

 \noindent\textbf{\textit{The Heisenberg group.}}
The Heisenberg group, $\mathbb{H}$,  can be described as $\mathbb{R}^3$ with the following group operation:
$$x\cdot\tilde{x} = (x_1, x_2, x_3)\cdot(\tilde{x}_1, \tilde{x}_2, \tilde{x}_3) = (x_1 + \tilde{x}_1, x_2 + \tilde{x}_2, x_3 + \tilde{x}_3 + \frac{1}{2}(x_1\tilde{x}_2 - x_2\tilde{x}_1))$$
$\mathbb{H}$ is a Lie group, and its Lie algebra $\mathfrak{h}$ can be identified with the space of left invariant vector fields on $\mathbb{H}$ in the standard way.  By direct computation we see that this space is spanned by 
\begin{eqnarray}
X_1 &=& \partial_{x_1} - \frac{1}{2}x_2\partial_{x_3} \nonumber \\
X_2 &=& \partial_{x_2} + \frac{1}{2}x_1\partial_{x_3} \nonumber \\
X_3 &=& \partial_{x_3} = [X_1, X_2] \nonumber
\end{eqnarray}
From this it is clear that $X_1, X_2$ satisfy the H\"ormander condition (i.e. $X_1, X_2$ and their commutator $[X_1, X_2]$ span the tangent space at every point of $\mathbb{H}$).  It is also easy to check that the left invariant Haar measure (which is also the right invariant measure since the group is nilpotent) is the Lebesgue measure $dx$ on $\mathbb{R}^3$. 
 
On $C^\infty_0(\mathbb{H})$, define the \textit{sub-gradient} to be the operator given by
$$\nabla := (X_1,X_2)$$ 
and the \textit{sub-Laplacian} to be the second order operator given by 
$$\Delta := X_1^2 + X_2^2$$
$\nabla$ can be treated as a closed operator from $L^2(\mathbb{H}, dx)$ to $L^2(\mathbb{H}; \mathbb{R}^2, dx)$.  Similarly, since $\Delta$ is densely defined and symmetric in $L^2(\mathbb{H}, dx)$, we may treat $\Delta$ as a closed self-adjoint operator on $L^2(\mathbb{H}, dx)$ by taking the Friedrich extension.  
We introduce the Logarithmic Sobolev Inequality on $\mathbb{H}$ in the following way.

~

\textit{The $q$ Log-Sobolev Inequality  (LSq) on $\mathbb{H}$.}
Let $q\in(1,2]$, and let $\mu$ be a probability measure on $\mathbb{H}$.  $\mu$ is said to satisfy a \textit{$q$ Logarithmic Sobolev inequality} $(LSq)$ on $\mathbb{H}$ if there exists a constant $c>0$ such that for all smooth functions $f: \mathbb{H}\to \mathbb{R}$
\begin{equation}
\label{eq21IP}
\mu\left(|f|^q\log\frac{|f|^q}{\mu|f|^q}\right) \leq c\mu\left(|\nabla f|^q\right)
\end{equation}
where $\nabla$ is the sub-gradient on $\mathbb{H}$.
\begin{Remark}
Since we have the sub-gradient on the right hand side, \eqref{eq21IP} is a Logarithmic Sobolev inequality corresponding to a H\"ormander type generator.  Indeed,  if $\mu(dx) = \frac{e^{-U}}{Z}dx$ then it is clear that $\mathcal{L} = \Delta - \nabla U.\nabla$ is a Dirichlet operator satisfying
\[
\mu\left(f\mathcal{L}f\right) = - \mu\left(|\nabla f|^2\right)
\]
where $\Delta$ is the sub-Laplacian, and $\nabla$ the sub-gradient.
\end{Remark}
In \cite{H-Z} the authors were able to show that a related class of measures on $\mathbb{H}$ satisfy $(LSq)$ inequalities (see Theorem \ref{thm223IP} below).  To describe these we first need to introduce the natural distance function on $\mathbb{H}$, which is the so-called \textit{Carnot-Carath\'eodory distance}.  This  distance is more natural than the usual Euclidean one, since it takes into account the extra structure that the Heisenberg group posseses.  
We define the Carnot-Carath\'eodory distance between two points in $\mathbb{H}$ by considering only \textit{admissible} curves between them in the following sense.  A Lipschitz curve $\gamma:[0,1]\to\mathbb{H}$ is said to be \textit{admissible} if $\gamma'(s) = a_1(s)X_1(\gamma(s)) + a_2(s)X_2(\gamma(s))$ almost everywhere with measurable coefficients $a_1, a_2$ i.e. if $\gamma'(s) \in sp\{X_1(\gamma(s)),X_2(\gamma(s))\}$ a.e.  Then the length of $\gamma$ is given by
$$l(\gamma) = \int_0^1\left(a_1^2(s) + a_2^2(s)\right)^{1/2}ds$$
and we define the Carnot-Carath\'eodory distance between two points $x,y\in\mathbb{H}$ to be 
$$d(x,y) := \inf\{l(\gamma):\gamma\ \textrm{is an admissible path joining}\ x\ \textrm{to}\ y\}.$$
Write $d(x)=d(x,e)$, where $e$ is the identity.

\begin{Remark}
This distance function is well defined as a result of Chow's theorem, which states that every two points in $\mathbb{H}$ can be joined by an admissible curve (see for example \cite{B-L-U},\cite{Grom}).
\end{Remark}

Geodesics are smooth, and are helices in $\mathbb{R}^3$.  
We also  have that $x=(x_1, x_2, x_3)\mapsto d(x)$ is smooth for $(x_1, x_2) \neq 0$, but not at points $(0,0, x_3)$, so that the unit ball has singularities on the $x_3$-axis.
In our analysis, we will frequently use the following two results.  The first is the well-known fact that the Carnot-Carath\'eodory distance satisfies the eikonal equation (see for example \cite{Mo}):
\begin{proposition}
\label{prop221IP}
Let $\nabla$ be the sub-gradient on $\mathbb{H}$.  Then $|\nabla d(x)|=1$ for all $x=(x_1, x_2, x_3)\in\mathbb{H}$ such that $(x_1, x_2)\neq0$.
\end{proposition}
We must be careful in dealing with the notion of $\Delta d$, since it will have singularities on the $x_3$-axis.  However, the following proposition from \cite{I-P} provides some control of these singularities.
\begin{proposition}
\label{prop222IP}
Let $\Delta$ be the sub-Laplacian on $\mathbb{H}$.
There exists a constant $K_{0}$ such that $\Delta d \leq \frac{K_{0}}{d}$ in the sense of distributions.
\end{proposition}

The following result concerning the q Log-Sobolev inequality can be   found in \cite{H-Z}.
\begin{theorem} \textbf{(\cite{H-Z})}
\label{thm223IP}
Let $\mu_p$ be the probability measure on $\mathbb{H}$ given by 
\[
\mu_p (dx) = \frac{e^{-\beta d^p(x)}}{\int_\mathbb{H} e^{-\beta d^p(x)}dx}dx
\]
where $p\geq2$, $\beta >0$, $dx$ is the Lebesgue measure on $\mathbb{R}^3$ and $d(x)$ is the Carnot-Carath\'edory distance.  Then $\mu_p$ satisfies an $(LS_q)$ inequality, where $\frac{1}{p}+\frac{1}{q} = 1$.
\end{theorem}
In order to prove the  Log-Sobolev inequality, the following inequality, called the U-bound, was first shown. 
$$\mu_p\left\vert f \right\vert^q d^{p}(x)\leq C\mu_p\ \left\vert \nabla f
\right\vert^q+D\mu_p\left\vert f \right\vert^q$$
for some constants $C,D>0$. 
More generally, for  arbitrary measures $$\mu (dx) = \frac{e^{-U(x)}}{\int e^{-U(x)}dx}dx$$  the authors in \cite{H-Z}  associated the q Log-Sobolev  Inequality with the following  U-bound inequality:
 \begin{align*}\mu\left\vert f \right\vert^q(\vert \nabla U\vert^q+U)\leq C\mu\ \left\vert \nabla f
\right\vert^q+D\mu\left\vert f \right\vert^q\end{align*} 
for constants $C,D>0$. Concerning the weaker Spectral Gap inequality the associated inequality is  \begin{align*}\mu\left\vert f \right\vert^q\eta\leq C'\mu\ \left\vert \nabla f
\right\vert^q+D'\mu\left\vert f \right\vert^q\end{align*} 
for some non negative non decreasing function $\eta$ and constants $C',D'>0$.

\subsection{Infinite dimensional setting.} In this section we present the infinite dimensional setting as well as past results for the infinite dimensional  Gibbs measure.
 
 ~

\textit{\textbf{Infinite dimensional analysis.}}             In the more standard Euclidean model the problem has been extensively discussed. Regarding the Log-Sobolev Inequality for the local specification
 $\{\mathbb{E}^{\Lambda,\omega}\}_{\Lambda\subset\subset \mathbb{Z}^d,\omega \in
\Omega}$ on a d-dimensional Lattice, criterions and examples of measures   $\mathbb{E}^{\Lambda,\omega}$ that satisfy
the Log-Sobolev -with a constant uniformly on the set $\Lambda$ and the boundary
conditions $\omega-$ are investigated in \cite{Z2}, \cite{B-E}, \cite{B-L}, \cite{Y} and \cite{B-H}.  For $\left\Vert \nabla_i \nabla_j V(x_i,x_j) \right\Vert_{\infty}<\infty$ the Log-Sobolev is proved  when the phase
$\phi$ is strictly convex and convex at infinity. Furthermore, in \cite{G-R} and \cite{B-Z}
the Spectral Gap Inequality is proved to be true for phases beyond the convexity at infinity.

For
the measure  $\mathbb{E}^{\{i\},\omega}$ on the real line, necessary
and sufficient  conditions are presented in  \cite{B-G}, \cite{B-Z} and \cite{R-Z}, so that
  the Log-Sobolev Inequality is satisfied uniformly on the boundary
conditions $\omega$.

The problem
of the Log-Sobolev inequality for the Infinite dimensional Gibbs measure on the
Lattice is examined
in \cite{G-Z}, \cite{Z1} and \cite{Z2}. The first two study the LS for measures
on a d-dimensional Lattice for bounded spin systems, while the third one
looks at continuous spins systems on the one dimensional Lattice.

In \cite{Ma} and \cite{O-R},  criterions are presented  in order to pass from the Log-Sobolev
Inequality for the single-site    measure $\mathbb{E}^{\{i\},\omega}$
to the (LS2) for the Gibbs measure $\nu_{N}$ on a finite N-dimensional product space. Furthermore, using these criterions  one can conclude  the Log-Sobolev Inequality for the family $\{\nu_N,N\in\mathbb{N}\}$ with a constant uniformly
on $N$. 

In \cite{L-Z} a similar situation is studied, in that the authors consider a system of H\"ormander generators in infinite dimensions and prove logarithmic Sobolev inequalities as well as some ergodicity results.  The main difference between the present set up and their situation is that we consider a non-compact underlying space, namely the Heisenberg group, in which the techniques of \cite{L-Z} cannot be applied.
  Concerning the same problem for the LSq ($q\in (1,2]$) inequality in the case of Heisenberg groups with
quadratic interactions in 
\cite{I-P}  a similar criterion is presented for the Gibbs  measure based on the methods
developed in \cite{Z1} and \cite{Z2}.

Our general setting is as follows:
 
 \textit{The Lattice.} When we  refer to the  Lattice we mean
 the 1-dimensional 
Lattice $\mathbb{Z}$. 
 
\textit{The Configuration space.}
We consider continuous unbounded random variables in $\mathbb{H}$, representing spins. Our configuration space is $\Omega=\mathbb{H}^{\mathbb{Z}}$. For any
$\omega\in \Omega $ and $\Lambda \subset \mathbb{Z}$ we denote
$$\omega=(\omega_i)_{i\in \mathbb{Z}}, \omega_{\Lambda}=(\omega_i)_{i\in\Lambda},\omega_{\Lambda^c}=(\omega_i)_{i\in\Lambda^c} \text{\; and \;}\omega=\omega_{\Lambda}\circ\omega_{\Lambda^c}$$
where $\omega_i \in \mathbb{H}$. When $\Lambda=\{i\}$ we will write $\omega_i=\omega_{\{i\}}$.
Furthermore, we will write $i\sim j$ when the nodes $i$ and $j$     are nearest neighbours, that means, they are connected with a vertex,
while
 we will denote the set of the neighbours of $k$ as   $\{\sim k\}=\{r:r \sim k\}$.
 
\textit{The functions of the configuration.}
We consider integrable functions $f$ that depend on a finite set of variables
$\{x_i\}, i\in{\Sigma_f}$ for a finite subset $\Sigma_f\subset\subset \mathbb{Z}$. The symbol $\subset\subset$ is used to denote a finite subset.

\textit{The Measure on $\mathbb{Z}$.}
For any subset $\Lambda\subset\subset
\mathbb{Z}$ we define the probability measure
$$\mathbb{E}^{\Lambda,\omega}(dx_\Lambda) =
\frac{e^{-H^{\Lambda,\omega}}dx_\Lambda} {Z^{\Lambda,\omega}}$$ 
  where
 \begin{itemize}
\item $x_{\Lambda}=(x_i)_{i\in\Lambda}$ and $dx_\Lambda=\prod_{i\in\Lambda}dx_i$

\item $Z^{\Lambda,\omega}=\int e^{-H^{\Lambda,\omega}}dx_\Lambda$
\item
$H^{\Lambda,\omega}=\sum_{i\in\Lambda}
\phi (x_i)+\sum_{i\in\Lambda,j\sim
i}J_{ij}V(x_i
,z_j)$ \item $\phi\geq 0$ and $V\geq 0$
 
\end{itemize}
\noindent
and
\begin{itemize}
\item
$z_j=x_{\Lambda}\circ\omega_{\Lambda^c}=\begin{cases}x_j & ,i\in\Lambda  \\
\omega_j & ,i\notin\Lambda \\   
\end{cases}$
\end{itemize}We call $\phi$ the phase and $V$ the potential of the interaction.
 For convenience we will frequently omit the boundary symbol from the measure and will
write $\mathbb{E}^{\Lambda}\equiv\mathbb{E}^{\Lambda,\omega}$.
 Furthermore we will assume that
 
 ~

   \noindent
\textbf{(H$^*$)} There exist   constants $B_*(L),B^*(L)\in (0,\infty)$ such that 
$$\int_{B_L\otimes B_L} e^{-H^{\{\sim i\},\omega}}dX_{\{\sim i \}}\geq\frac{1}{B^*(L)} \text{\; for \;} \omega_{j}\in B_L,j \in \{ i-2,i,i+2 \} $$and $$ e^{-H^{\{\sim i\},\omega}}\geq\frac{1}{B_*(L)}$$ $$\text{\;for \ }\omega_{j}\in B_L,j \in \{ i-2,i,i+2 \}  \text{\  and \ }x_j\in B_L,j\in\{i-1,i+1\} $$where for any $R\geq 0$, $B_R=\{x\in \mathbb{H}:d(x)\leq R\}$.

 ~

 \begin{Remark} The hypothesis $(H^*)$ is a technical condition which essentially does not allow  singularities on $\phi$ and $V$. 
 \end{Remark}  
\textit{The  Infinite Volume Gibbs Measure.}   The Gibbs measure $\nu$ for the local specification
$\{\mathbb{E}^{\Lambda,\omega}\}_{\Lambda\subset \mathbb{Z},\omega \in \Omega}$ is defined as the  probability measure which solves the Dobrushin-Lanford-Ruelle (DLR) equation
 $$\nu \mathbb{E}^{\Lambda,\star}=\nu $$
for  finite  sets $\Lambda\subset \mathbb{Z}$ (see \cite{Pr}). For conditions on
the existence and uniqueness of the Gibbs measure see e.g. \cite{B-HK}
and \cite{D}.  It should be noted that $\{\mathbb{E}^{\Lambda,\omega}\}_{\Lambda\subset\subset \mathbb{Z},\omega \in
\Omega}$ always satisfies the DLR equation, in the sense that
$$\mathbb{E}^{\Lambda,\omega}\mathbb{E}^{M,\ast}=\mathbb{E}^{\Lambda,\omega}$$ for every $M\subset\Lambda$. \cite{Pr}.

\textit{The gradient $\nabla$ for continuous spins systems.}
\noindent
For any subset $\Lambda\subset \mathbb{Z}$ we define the gradient
 $$\left\vert \nabla_{\Lambda} f
\right\vert^q=\sum_{i\in\Lambda}\left\vert \nabla_{i} f
\right\vert^q $$the \textit{sub-gradient} $\nabla_i$ corresponds to the i'th variable
$$\nabla_{i} := (X^{i}_1,X^{i}_2)$$ 
When $\Lambda=\mathbb{Z}$ we will simply write $\nabla=\nabla_\mathbb{Z}$. We denote
$$\mathbb{E}^{\Lambda,\omega}f=\int f d\mathbb{E}^{\Lambda,\omega}(x_\Lambda)$$
 Under this specific higher dimensional setting the  q Logarithmic Sobolev and q Spectral Gap inequalities  are defined for measures of the local specification $\{\mathbb{E}^{\Lambda,\omega} \}_{\Lambda\subset \subset \mathbb{Z},\omega\in\Omega}$.
 
\textit{The $q$ Log-Sobolev Inequality  (LSq) on $\mathbb{H}^\mathbb{Z}$.}
 We say that the measure $\mathbb{E}^{\Lambda,\omega}$ satisfies the q Log-Sobolev  Inequality for $q\in (1,2]$, if there exists a constant $C_{LS}$   such that for any function $f$, the following  holds
$$\mathbb{E}^{\Lambda,\omega}\left\vert f\right\vert^qlog\frac{\left\vert f\right\vert^q}{\mathbb{E}^{\Lambda,\omega}\left\vert f\right\vert^q}\leq C_{LS}
\mathbb{E}^{\Lambda,\omega}\left\vert \nabla_{\Lambda} f
\right\vert^q$$  
with the constant $ C_{LS}\in(0,\infty)$  uniformly on the set $\Lambda$  and the boundary conditions
$\omega$.
 
\textit{The $q$ Spectral Gap Inequality on $\mathbb{H}^\mathbb{Z}$.} We say that the measure $\mathbb{E}^{\Lambda,\omega}$ satisfies the q Spectral Gap  Inequality for $q\in (1,2]$, if there exists a constant $C_{SG}$   such that for any function $f$, the following  holds
$$\mathbb{E}^{\Lambda,\omega}\left\vert f-\mathbb{E}^{\Lambda,\omega}f \right\vert^q\leq C_{SG}
\mathbb{E}^{\Lambda,\omega}\left\vert \nabla_{\Lambda} f
\right\vert^q$$
with the constant $C_{SG}\in(0,\infty)$  uniformly on the set $\Lambda$  and the boundary conditions
$\omega$.
\begin{Remark}\label{rem1.1} We will frequently use the following two well known properties about the
 Log-Sobolev  and the Spectral Gap Inequality. If the probability measure $\mu$ satisfies
the Log-Sobolev Inequality with constant $c$ then it also satisfies the Spectral Gap Inequality with a constant   $\hat c= \frac{4 c}{\log 2}$.
More detailed, in the case where $q=2$ the optimal constant is less or equal to $\frac{c}{2}<\hat c$, while  in the case $1<q<2$ it is less or equal to $\frac{4c}{\log 2}$. The constant $\hat c$ does not depend  on the value of  the parameter $q\in(1,2]$.  

Furthermore, if for a family $I$
of sets  $\Lambda_i \subset \mathbb{Z}$,
$dist(\Lambda_i,\Lambda_j)>1 \ , i\neq j$    the measures $\mathbb{E}^{\Lambda_{i},\omega},
i\in I$
satisfy the Log-Sobolev Inequality with constants $c_i,i\in I$,  then the
probability measure $\mathbb{E}^{\{\cup_{i\in I}\Lambda_i\},\omega}$ also satisfies
the (LS) Inequality with constant $c=max_{i\in I} c_i$. The last result is also true for the Spectral Gap Inequality. The
 proofs of these two properties can be found in  \cite{Gros} and   \cite{G-Z} for $q=2$
 and in \cite{B-Z} for $1<q<2$. \end{Remark}
  Concerning the $q$ Log-Sobolev inequality for spins on the Heisenberg group, in \cite{I-P} the inequality was proven for a specific class of H\"ormander type  generators on the Heisenberg group. The main hypothesis for the local specification was that   $\Vert V''\Vert_\infty<\infty$ and that the one dimensional measures $\mathbb{E}^{i,\omega}$ satisfies
the Log-Sobolev-q Inequality with a constant $c$ uniformly with  respect to the boundary
conditions $\omega$.    Furthermore   the main hypothesis  that  $\mathbb{E}^{i,\omega}$ satisfies
the Log-Sobolev-q Inequality with a constant $c$ uniformly with respect to the boundary
conditions $\omega$ was proven to hold for a specific class of  local specification  
with  
\begin{equation}
\label{eq25PapIPn}
H^{\Lambda,\omega}(x_\Lambda) =\\\alpha \sum_{i\in \Lambda}d^p(x_i) + \varepsilon\sum_{\{i,j\}\cap\Lambda \not=\emptyset ,j:j\sim i}(d(x_i) + \rho d(\omega_j))^{2}  
\end{equation}
for $\alpha>0$, $\varepsilon, \rho \in\mathbb{R}$, and $p\geq2$ (with $\varepsilon>-\frac{\alpha}{2N}$ if $p=2$ for), where as above  $x_i = \omega_i$ for $i\not\in \Lambda$.
 The proof is based on showing the following U-bound
$$\mathbb{E}^{i,\omega}\left(|f|^q\left(d^{p-1}  + \sum_{j:j\sim i}d(\omega_j)\right)\right) \leq A_1\mathbb{E}^{i,\omega}|\nabla_if|^q + B_1\mathbb{E}^{i,\omega} |f|^q$$
for all smooth $f:\Omega\to\mathbb{R}$, and some constants $A_1, B_1\in(0,\infty)$ independent of $\omega$ and $i$.
 One of the purposes of the current paper is to present criterions which will allow to obtain the (LSq) inequality for the infinite dimensional Gibbs measure in the case where the measures  $\mathbb{E}^{i,\omega}$ satisfy a U-bound inequality as above with a constant $B_1$ which is not independent of the boundary conditions $\omega$.

If one starts with the weaker condition, that the one dimensional single site measures $\nu(dx_{i})=\frac{e^{-\phi(x_{i})}dx_i}{\int e^{-\phi(x_{i})}dx_i}$ satisfy an (LS2)  inequality 
 then as shown in \cite{Pa2} we can obtain a Talagrand type inequality for the Gibbs measure, similar to the ones obtained in \cite{B-R2}.

In [30], although the Log-Sobolev inequality was not eventually obtain  for the infinite dimensional Gibbs measure, it was showed that concentration properties still hold true. However, these are weaker than the ones that hold  for the product measure associated with the Log-Sobolev inequality, and similar to the ones that the Gibbs measure satisfies in the case of the Modified Log-Sobolev inequality $MLS(H_\Phi)$ with $\Phi(x)=x^4$ (see  \cite{B-R1},\cite{B-R2},\cite{G-G-M1},\cite{G-G-M2} and \cite{Pa3}).
In this paper we will see that under  appropriate c££££onditions  the actual (LSq) inequality can be obtained for the Gibbs measure. 

All the pre-mentioned developments refer to measures with interactions $V$
that satisfy 
$\left\Vert \nabla_i \nabla_j V(x_i,x_j) \right\Vert_{\infty}<\infty$. The question
that arises is whether similar assertions can be  verified for the infinite
dimensional Gibbs
measure in the case where  $\left\Vert \nabla_i \nabla_j V(x_i,x_j) \right\Vert_{\infty}= \infty$. In \cite{Pa1}, under the main hypothesis that  $\mathbb{E}^{i,\omega}$ satisfies
the Log-Sobolev-q Inequality with a constant $c$ uniformly with respect to the boundary
conditions $\omega$, such a result  was presented under the three   hypothesis:
 
~

\noindent
\textbf{(C1)}: The restriction  $\nu_{\Lambda(k)}$ of the Gibbs measure $\nu$
to the $\sigma-$algebra $\Sigma_{\Lambda(k)}$,    $$\Lambda(k)=\{k-2,k-1,k,k+1,k+2\}$$
\ \ \ \   \ \ \ \ \  \ satisfies the Log-Sobolev-q Inequality with a constant $C\in (0,\infty)$.
 
 ~

 \noindent
\textbf{(C2)}:
For
some $\epsilon>0$ and $\hat K>0$ such that  $$\nu_{\Lambda(i)} e^{2^{q+2}\epsilon 
V(x_r,x_s)}\leq \hat K\text{\;, \;} \nu_{\Lambda(i)} e^{2^{q+2}\epsilon \left\vert\nabla_{r}V(x_{r},x_{s})\right\vert^q}\leq\hat K \text{\; and \; }\nu_{\Lambda(i)}e^{\epsilon 2^{10} d(x_i)}<\hat K$$for $r,s\in\{i-2,i-1,i,i+1,i+2\}$.

~

\noindent
\textbf{(C3)}:
The coefficients $J_{i,j}$ are such that $\left\vert J_{i,j}\right\vert\in[0,J]$
for some $J<1$ sufficiently small.
  \begin{theorem}\label{thm2.1}\textbf{(\cite{Pa1})} Assume that  the one dimensional measures $\mathbb{E}^{i,\omega}$ satisfies
the Log-Sobolev-q Inequality with a constant $c$ uniformly with respect to the boundary
conditions $\omega$.
If hypothesis  (C1),   (C2) and (C3) are satisfied, then,  the infinite dimensional Gibbs measure $\nu$  for the local specification $\{\mathbb{E}^{\Lambda,\omega}\}_{\Lambda\subset\subset \mathbb{Z},\omega \in
\Omega}$ satisfies the $q$ Log-Sobolev  inequality
$$\nu \left\vert f\right\vert^q log\frac{\left\vert f\right\vert^q}{\nu \left\vert f\right\vert^q}\leq \mathfrak{C} \ \nu \left\vert \nabla f
\right\vert^q$$                              
for some positive constant $\mathfrak{C}$. \end{theorem}
 As an  example of a measure  $\mathbb{E}^{i,\omega}$ that satisfies
the Log-Sobolev-q Inequality with a constant $c$ uniformly with respect to the boundary
conditions $\omega$ with non quadratic interaction on the Heisenberg group one can think of a measure similar with that on (\ref{eq25PapIPn}) but with interactions of higher growth, i.e. 
\begin{equation*}
H^{\Lambda,\omega}(x_\Lambda) =\\\alpha \sum_{i\in \Lambda}d^p(x_i) + \varepsilon\sum_{\{i,j\}\cap\Lambda \not=\emptyset ,j:j\sim i}(d(x_i) + \rho d(\omega_j))^{s}  
\end{equation*}
for $\alpha>0$, $\varepsilon, \rho \in\mathbb{R}$, and $p>s>2$, where as above  $x_i = \omega_i$ for $i\not\in \Lambda$. The proof of this follows  with the use of uniform U-Bounds (see (\cite{Pa1})). One of the purposes of the current paper is to relax the main hypothesis that $\mathbb{E}^{i,\omega}$ satisfies
the Log-Sobolev-q Inequality with a constant $c$ uniformly on the boundary conditions. Furthermore, a simplification of the last 
theorem will be obtained in the next section in Theorem \ref{MarkovSimplif}. 
\section{Main results.}

We focus on the Logarithmic Sobolev Inequality (LS$q$) for measures related to systems with values in the Heisenberg group on the one dimensional Lattice with nearest neighbour interactions. The aim  is to investigate the conditions under which the  inequality  can be extended from the one dimensional measure to  the  Infinite  volume Gibbs measure. 

In this paper we apply the same ideas as in \cite{H-Z}, \cite{I-P} and \cite{Pa1} to investigate cases of measures were the U-bound inequalities do not hold uniformly on the boundary conditions but still the infinite volume Gibbs measure ultimately satisfies the Log-Sobolev inequality.  We will be concerned with two cases.

~

\noindent
\textbf{Case 1: A Perturbation property.}  
 The first case is  actually  a perturbation result on the measures obtained in \cite{H-Z}. We recall that  according to \cite{H-Z}, the measure on $\mathbb{H}$  given by 
\[
\mu_p (dx) = \frac{e^{-\beta d^p(x)}}{\int_\mathbb{H} e^{-\beta d^p(x)}dx}dx
\]
where $p\geq2$, $\beta >0$,  satisfies an $(LS_q)$ inequality, where $\frac{1}{p}+\frac{1}{q} = 1$. We try to address the following question. If we perturb this measure with interactions to obtain the following local specifications  $$\mathbb{E}^{\{i\},\omega}(dx_{i})=\frac{e^{-\beta d^p(x)-\sum_{j\sim
i}J_{ij}V(x_i
,\omega_j)}dX_i} {Z^{\{i\},\omega}} \text{\; with \;} \left\Vert \nabla_i \nabla_j V(x_i,x_j) \right\Vert_{\infty}\leq\infty$$          
 under which conditions does the        infinite volume
Gibbs measure  $\nu$ for the local specification $\{\mathbb{E}^{\Lambda,\omega}\}_{\Lambda\subset\subset \mathbb{Z},\omega \in
\Omega}$    satisfies the Log-Sobolev inequality?

In both \cite{I-P} and \cite{Pa1}, the main  assumption was that the one dimensional measures $\mathbb{E}^{i,\omega}$ satisfies
the Log-Sobolev-q Inequality with a constant $c$ uniformly with respect to the boundary
conditions $\omega$.
  In  this paper we want to relax  the main hypothesis  
for $\mathbb{E}^{\{i\},\omega}$ to the same assumption for the boundary free one dimensional measure.
In other words we want to address the following problem.

 Consider the local specification
\begin{equation}\label{eq218Papchap6++}\mathbb{E}^{\Lambda,\omega}(dx_\Lambda)=\frac{e^{-
\sum_{i \in \Lambda} \phi(x_i)-\sum_{i \in \Lambda}\sum_{j\sim
i}J_{ij}V(x_i,\omega_j)}dX_\Lambda} {Z^{ \Lambda,\omega}} \text{\; with \;} \left\Vert\partial_x \partial_y V(x,y) \right\Vert_{\infty} \leq \infty\end{equation}               and assume  that

       ~

\noindent
\textbf{(H1)}: The one site  measures $\mu(dx_{i})=\frac{e^{-\phi (x_{i})}dx_{i}}{\int e^{-\phi (x_{i})}dx_{i}}$ satisfies
the q Log-Sobolev  Inequality with a constant $c$.
 
 ~

 Under which conditions does the        infinite volume
Gibbs measure  $\nu$ corresponding to the local specification $\{\mathbb{E}^{\Lambda,\omega}\}_{\Lambda\subset\subset \mathbb{Z},\omega \in
\Omega}$   of (\ref{eq218Papchap6++}) satisfies the Log-Sobolev inequality? We present a strategy to solve this problem. As we will see, hypothesis (H1), together with (C1) imply the $q$ Log-Sobolev inequality for the infinite dimensional Gibbs measure. We will focus on  measures
on the  one dimensional lattice,  but our result can also be easily extended
on trees.

~

\noindent
\textbf{Case 2: Non uniform U-Bound.} As explained in the introduction, the U-bound inequalities introduced in \cite{H-Z} are an essential tool in proving the Spectral Gap and the Logarithmic Sobolev inequality, under the framework of the Heisenberg group. In the case of the specific example examined in \cite{I-P}, for the proof of both the Spectral Gap and the Log-Sobolev inequality the basic step was again the  U-bound inequalities. In order to obtain the two coercive inequalities uniformly on the boundary conditions, the two U-bounds had to be proven to hold also independently of the boundary conditions of the measure $\mathbb{E}^{i,\omega}$.

  Here we investigate cases were   weaker U-bound inequalities hold for  $\mathbb{E}^{i,\omega}$. 
In particular we concentrate on these cases were one of the constants depends on the boundary conditions $\omega$. For the  local specification 
\begin{equation}\label{eq218Papchap7++}\mathbb{E}^{\Lambda,\omega}(dx_\Lambda)=\frac{e^{-
H^{\Lambda,\omega}}dX_\Lambda} {Z^{ \Lambda,\omega}} \end{equation}
for $\Lambda\subset\subset\mathbb{Z}$ and $\omega\in \Omega$, with $$H^{\Lambda,\omega}=\sum_{i \in \Lambda} \phi(x_i)+\sum_{i \in \Lambda}\sum_{j\sim
i}J_{ij}V(x_i,\omega_j)\text{\; with \;} \left\Vert\partial_x \partial_y V(x,y) \right\Vert_{\infty}\leq \infty$$
we consider the  following hypothesis:

~

 \noindent
\textbf{(H2)}: \textit{Non uniform U-bound. }
 \begin{align*}\mathbb{E}^{\{\sim i\},\omega}\left\vert f \right\vert^q(\vert \nabla_{\{\sim i\}}H^{\{\sim i\},\omega}\vert^q+H^{\{\sim i\},\omega})\leq \hat C\mathbb{E}^{\{\sim i\},\omega}\ \left\vert \nabla_{\{\sim i\}} f
\right\vert^q+\hat D_{\{\sim i\}}(\omega)\mathbb{E}^{\{\sim i\},\omega}\left\vert f \right\vert^q\end{align*}for functions $f\in C^{\infty}$ for which the right-hand side is well defined, with $\nu e^{\epsilon\hat D_{\{\sim i\}}(\omega)}\leq \hat K$ where $\hat D_{\{\sim i\}}(\omega)$ is a function of $\omega_{i-2},\omega_i,\omega_{i+2}$.
What we will show is that even when the Log-Sobolev inequality  does not hold for the one site measure    $\mathbb{E}^{i,\omega}$ with a constant uniformly on the boundary, we can still obtain the inequality for the infinite dimensional Gibbs measure. 

Before we present the  main result a  remark concerning the conditions will follow.
\begin{Remark}\label{paradigmU-bounds}

 For examples of measures $$\mathbb{E}^{i,\omega}(dx_{i})=\frac{e^{-\phi(x_i)-\sum_{j\sim
i}J_{ij}V(x_i
,\omega_j)}dx_i} {\int e^{-\phi(x_i)-\sum_{j\sim
i}J_{ij}V(x_i
,\omega_j)}dx_i}$$ on the Heisenberg group, that satisfy  the non uniform U-bound   (H2) one can think of the following two

~

\noindent
 (i) $\phi(x)=d(x)^s$ for $0 \leq s <2 $ and 
 $V(x,y)=\left(d(x)-d(y)\right)^2$ in the case of (LS2)
 
~

and

~

\noindent (ii)  $\phi(x)=x^s$ for $0 \leq s <p $ and   $V(x,y)=\left(d(x)+d(y)\right)^p$,  where $\frac{1}{p}+\frac{1}{q}=1$ in the case of (LSq).

\end{Remark}
The main  theorem    follows.

\begin{theorem}\label{newmaintheorem} Assume that for  the local specification  $\{\mathbb{E}^{\Lambda,\omega}\}_{\Lambda\subset\subset \mathbb{Z},\omega \in 
\Omega}$   either hypothesis (H1) or (H2) is true and that  the interactions are such that: \begin{align}\label{newinteractionbound} \left \vert \nabla_i V \right \vert^q \leq a V+b
\end{align}
for positive constants $a,b$. Then for  coefficients $J_{i,j}$   sufficiently small, the infinite dimensional Gibbs measure   $\nu $  satisfies the Log-Sobolev q inequality
$$\nu \left\vert f\right\vert^q log\frac{\left\vert f\right\vert^q}{\nu \left\vert f\right\vert^q}\leq \mathfrak{C} \ \nu \left\vert \nabla f
\right\vert^q$$                              
for some positive constant $\mathfrak{C}\in (0,\infty$) if and only if (C1) is satisfied.
\end{theorem}

 The proof of Theorem \ref{newmaintheorem}  will be based on the  following weaker but more general result.\begin{theorem}\label{thm2.1Pap2} Assume that either hypothesis (H1) or (H2) is true. If the conditions  (C1), (C2) and (C3) for the local specification  $\{\mathbb{E}^{\Lambda,\omega}\}_{\Lambda\subset\subset \mathbb{Z},\omega \in 
\Omega}$ are satisfied, then the infinite dimensional Gibbs measure $\nu$  for the local specification $\{\mathbb{E}^{\Lambda,\omega}\}_{\Lambda\subset\subset \mathbb{Z},\omega \in
\Omega}$ satisfies the Log-Sobolev q inequality
$$\nu \left\vert f\right\vert^q log\frac{\left\vert f\right\vert^q}{\nu \left\vert f\right\vert^q}\leq \mathfrak{C} \ \nu \left\vert \nabla f
\right\vert^q$$                              
for some positive constant $\mathfrak{C}\in (0,\infty)$.     \end{theorem}
\begin{Remark} Theorem \ref{thm2.1Pap2} is more general that Theorem \ref{newmaintheorem}, since, if we consider the (LSq) under hypothesis (H1),  it also applies to local specifications that do not satisfy the bound (\ref{newinteractionbound}), e.g. $V(x_i,x_j)=\left \vert x_i-x_j\right \vert^s$ with $s>p$ where $\frac{1}{p}+\frac{1}{q}=1$, if the phase is $\phi(x)=\left \vert x \right \vert^r$ with $r>\max\{s,(s-1)q\}$. 
\end{Remark}
\noindent
\textit{\textbf{Proof of Theorem~\ref{newmaintheorem}.}} The proof of Theorem \ref{newmaintheorem} follows directly from  Theorem \ref{thm2.1Pap2} since hypothesis (C2) is always satisfied when (C1) and (\ref{newinteractionbound}) are true, as explained in the following theorem due to Hebisch and Zegarlinski (see \cite{H-Z}) 
\begin{theorem}\label{boundThm(H-Z)}\textbf{(\cite{H-Z})} Assume that a measure $\mu$ satisfies (LSq) for some $q\in(1,2]$.  Suppose that for some  constants $a,b\in(0,+\infty)$ we have \begin{align*}\label{newinteractionbound} \left \vert \nabla  f \right \vert^q \leq af+b
\end{align*}
Then the following exp-bound is true
\begin{align*} \mu e^{tf}<\infty
\end{align*} for all $t>0$ sufficiently small.
\end{theorem} 
\qed

As a matter of fact the same argumentation can be used to simplify
Theorem \ref{thm2.1},  in which case we obtain the following result 
\begin{theorem}\label{MarkovSimplif} Assume that  the one dimensional measures $\mathbb{E}^{i,\omega}$ satisfy
the Log-Sobolev-q Inequality with a constant $c$ uniformly with respect to the boundary
conditions $\omega$
and that  the interactions are such that: \begin{align*} \left \vert \nabla_i V \right \vert^q \leq a V+b
\end{align*}
for positive constants $a,b$. Then for  coefficients $J_{i,j}$   sufficiently small, the infinite dimensional Gibbs measure   $\nu $  satisfies the Log-Sobolev q inequality
$$\nu \left\vert f\right\vert^q log\frac{\left\vert f\right\vert^q}{\nu \left\vert f\right\vert^q}\leq \mathfrak{C} \ \nu \left\vert \nabla f
\right\vert^q$$                              
for some positive constant $\mathfrak{C}\in (0,\infty$) if and only if (C1) is satisfied.
\end{theorem}
 The rest of the paper will be dedicated in proving  Theorem~\ref{thm2.1Pap2}.  Aside from hypothesis (H1) and (H2)  the rest of the assumptions are 
 the same as in \cite{Pa1}. 

\begin{Remark}\label{rem2.1nPap1} 
  From Hypothesis $(C2)$ and H\"older inequality
 it follows that   $$\nu e^{\epsilon(\left\vert F(r)\right\vert+\mathbb{E}^{S(r),\omega}\left\vert F(r)\right\vert)^q}\leq \hat
K,\text{\; for \;}r=i-2,i-1,i,i+1,i+2$$where the functions $F(r)$ are defined by
$$ F(r)=\begin{cases}\nabla_{r}V(x_{i-1},x_{i})+\nabla_{r}V(x_{i+1},x_{i})& \text{\; for \;}r=i-1,i,i+1  \\
\nabla_{r}V(x_s,x_r)\mathcal{I}_{\{s\sim r:s\in\{i-3,i+3\}\}} & \text{\; for \;}r=i-2,i+2  \\
\end{cases}
$$
 and the sets $S(r)$  by 
 $$S(r)=\begin{cases}\{\sim i\} & \text{\; for \;}r=i-1,i,i+1 \\
\{i+3,i+4,...\} & \text{\; for \;}r=i+2 \text{\;and\;}s=i+3\\
\{...,i-4,i-3\} & \text{\; for \;}r=i-2 \text{\;and\;}s=i-3 \\
\end{cases}$$ These bounds will be frequently  used through out this and the next chapter.
\end{Remark} For computational reasons we set $\hat K:=e^K$ and  $$\eta(i,\omega)=d( x_{i-1} )+d( x_{i+1} )+\sum_{j\sim\{\ i-1,i+1\}}d(\omega_{j} )$$

  In order to prove Theorem~\ref{thm2.1Pap2}   we will use the methods developed by Zegarlinski in \cite{Z1} and \cite{Z2}. The main idea is based
on approximating the infinite dimensional  Gibbs measure $\nu$ for the local specification $\{\mathbb{E}^{\Lambda,\omega}\}_{\Lambda\subset\subset \mathbb{Z},\omega \in
\Omega}$ by a sequence   which involves components in the local specification that satisfy
the Log-Sobolev inequality. This method was used in \cite{Pa1}  and \cite{I-P} where the one dimensional
measures $\mathbb{E}^{\{i\},\omega}$ satisfied the Log-Sobolev inequality
uniformly on the boundary conditions $\omega$.   In the two cases examined here
where either  the one dimensional boundary-free measure $\mu(dx_{i})=\frac{e^{-\phi (x_{i})}dx_{i}}{\int e^{-\phi (x_{i})}dx_{i}}$ satisfies an (LSq) or  $\mathbb{E}^{\{i\},\omega}$ a non-uniform U-bound, we will
 replace under our assumptions (H1) and (H2) together with (C1)-(C3), the    property  $\mathbb{E}^{\{i\},\omega}\in LSq$ uniformly on $\omega$, by a similar but weaker inequality that maintains
most of the properties of the Log-Sobolev inequality. This Log-Sobolev type inequality will be\begin{equation}\label{2.1Pap2}\nu\mathbb{E}^{ \{\sim i\},\omega}(\left\vert f\right\vert^qlog\frac{\left\vert f\right\vert^ q}{\mathbb{E}^{ \{\sim i\},\omega}\left\vert f\right\vert^q}) \leq R\sum_{r=i-2}^{i+2}\nu\left\vert \nabla_{r}
f
\right\vert^q+R\sum_{ r=3 }^{\infty} J_0^{r-2}\nu\left\vert \nabla_{i\pm r} f
\right\vert^q\end{equation} We will prove a similar inequality to replace  the Spectral Gap inequality. This will be \begin{equation}\label{2.2Pap2}\nu\mathbb{E}^{ \{\sim i\},\omega}\left\vert f-\mathbb{E}^{ \{\sim i\},\omega}f\right\vert^q\leq  M\sum_{r=i-2}^{i+2}\nu\left\vert \nabla_{r}
f
\right\vert^q+M\sum_{ r=3 }^{\infty} J_0^{r-2}\nu\left\vert \nabla_{i\pm r} f
\right\vert^q\end{equation} where $J_0<1$ is a constant depending on $J$. The q Log-Sobolev type inequality (\ref{2.1Pap2}) will be   shown in Proposition~\ref{prp3.4Pap2} and Proposition \ref{lem1.7P5},    under the hypothesis (H1) and (H2) respectively. The q Spectral Gap type inequality (\ref{2.2Pap2}) will be   proven in Proposition~\ref{prp3.2Pap2}     for both the cases of  hypothesis (H1) and (H2). In addition, 
 an analogue   of the product property for the Log-Sobolev inequality is proven in   Proposition ~\ref{prp2.2Pap2} for the inequality~\eqref{2.1Pap2}. The proof of
 Theorem ~\ref{thm2.1Pap2} follows.

 ~

\noindent
\textit{\textbf{Proof of Theorem~\ref{thm2.1Pap2}.}}  We want to extend the Log-Sobolev Inequality from the one site  measure  to the infinite dimensional Gibbs measure for the local specification $\{\mathbb{E}^{\Lambda,\omega}\}_{\Lambda\subset\subset \mathbb{Z},\omega \in
\Omega}$  on   the entire one dimensional lattice.      Define the following sets
$$\Gamma_0=   \text{\;even integers, \;} \Gamma_1=\mathbb{Z}\smallsetminus\Gamma_0$$
One can notice that $\{dist(i,j)>1, \ \forall i,j \in\Gamma_k,k=0,1\}$, $\Gamma_0\cap\Gamma_1=\emptyset$  and    $\mathbb{Z}=\Gamma_0\cup\Gamma_1$. For convenience we will write $\mathbb{E}^{\Gamma_i}=\mathbb{E}^{\Gamma_i,\omega}$ for $i=0,1$. Denote
\begin{equation}\label{defPch6} \mathcal{P}=\mathbb{E}^{\Gamma_1}\mathbb{E}^{\Gamma_{0}}\end{equation}
   In order to prove the Log-Sobolev Inequality for the measure $\nu$, we will express the entropy with respect to the measure $\nu$ as the sum of the entropies of the measures   $\mathbb{E}^{\Gamma_0}$  and $\mathbb{E}^{\Gamma_1}$.
 Assume $f\geq 0$. We can write  \begin{align}\nonumber \nu (f^q log\frac{f^q}{\nu f^q})=&\nu\mathbb{E}^{\Gamma_0} (f^q log\frac{f^q}{\mathbb{E}^{\Gamma_0} f^q})+\nu\mathbb{E}^{\Gamma_{1}} (\mathbb{E}^{\Gamma_0}f^q log\frac{\mathbb{E}^{\Gamma_0}f^q}{\mathbb{E}^{\Gamma_{1}}\mathbb{E}^{\Gamma_0} f^q})+\\ & \label{2.3Pap2}
\nu (\mathbb{E}^{\Gamma_{1}}\mathbb{E}^{\Gamma_{0}}f^q log\mathbb{E}^{\Gamma_{1}}\mathbb{E}^{\Gamma_{0}}f^q)-\nu
 (f^q log \nu f^{q})\end{align}
  The following proposition gives a Log-Sobolev type inequality for the product
  measures  $\mathbb{E}^{ \Gamma_{k},\omega},k=0,1$. \begin{proposition} \label{prp2.2Pap2}
   Assume  either  (H1) or (H2). If the conditions  (C1), (C2) and (C3) are satisfied then  the following
Log-Sobolev type inequality holds  
 $$\nu\mathbb{E}^{ \Gamma_{k}}(\left\vert f\right\vert^qlog\frac{\left\vert f\right\vert^q}{\mathbb{E}^{ \Gamma_{k}}\left\vert f\right\vert^q}) \leq\tilde C \nu\left\vert \nabla_{\Gamma_0}
f
\right\vert^q+\tilde C\nu\left\vert \nabla_{\Gamma_1} f
\right\vert^q$$for $k=0,1$, and  some positive constant $\tilde C$.
\end{proposition} The proof of Proposition~\ref{prp2.2Pap2} will be the subject of Section \ref{section6}.
If we use the Proposition~\ref{prp2.2Pap2} for $\mathbb{E}^{\Gamma_i},i=0,1$,  we get
 \begin{align}\nonumber ~\eqref{2.3Pap2}\leq &\tilde C \nu\left\vert \nabla_{\Gamma_0} f
\right\vert^q+ \tilde C \nu\left\vert \nabla_{\Gamma_1} f
\right\vert^q+\tilde C\nu\left\vert \nabla_{\Gamma_1}(\mathbb{E}^{\Gamma_0} f^q)^{\frac{1}{q}}
\right\vert^q\\ &\label{2.4Pap2}+\nu (\mathbb{E}^{\Gamma_{1}}\mathbb{E}^{\Gamma_{0}}f^q log\mathbb{E}^{\Gamma_{1}}\mathbb{E}^{\Gamma_{0}}f^q)-\nu (f^q log \nu f^q)\end{align}
For the fourth term on the right hand side  of~\eqref{2.4Pap2} we can write
\begin{align}\nonumber\nu (\mathcal{P}f^q log \mathcal{P} f^q)=&\nu \mathbb{E}^{\Gamma_{0}}(\mathcal{P}f^q log \frac{\mathcal{P} f^q}{\mathbb{E}^{\Gamma_{0}}\mathcal{P} f^q})+\nu \mathbb{E}^{\Gamma_{1}}(\mathbb{E}^{\Gamma_{0}}\mathcal{P}f^q log \frac{\mathbb{E}^{\Gamma_{0}}\mathcal{P} f^q}{\mathbb{E}^{\Gamma_{1}}\mathbb{E}^{\Gamma_{0}}\mathcal{P} f^q})\\  & \nonumber
+\nu( \mathbb{E}^{\Gamma_{1}}\mathbb{E}^{\Gamma_{0}}\mathcal{P}f^q log \mathbb{E}^{\Gamma_{1}}\mathbb{E}^{\Gamma_{0}}\mathcal{P}f^q )\end{align}
If we use again Proposition~\ref{prp2.2Pap2}  for the measures $\mathbb{E}^{\Gamma_{i}},i=0,1$
we get
 \begin{align}\nonumber\nu (\mathcal{P}f^qqlog \mathcal{P} f^q)\leq & \tilde C\nu\left\vert \nabla_{\Gamma_0}(\mathcal{P}f^q)^\frac{1}{q}
\right\vert^q\\   &\label{2.5Pap2}+\tilde C \nu\left\vert \nabla_{\Gamma_1}(\mathbb{E}^{\Gamma_{0}} \mathcal{P}f^q)^\frac{1}{q}
\right\vert^q+\nu (\mathcal{P}^2f^q log \mathcal{P}^2f^q) \end{align}
 If we work similarly for the last term $\nu (\mathcal{P}^2f^q \log \mathcal{P}^2f^q)$  of ~\eqref{2.5Pap2} and  inductively for any term  $\nu (\mathcal{P}^kf^q log \mathcal{P}^kf^q)$, then after $n$ steps~\eqref{2.4Pap2} and~\eqref{2.5Pap2} will give
\begin{align}\nonumber\nu (f^q log\frac{f^q}{\nu f^q})\leq &\nu (\mathcal{P}^n f^q log\mathcal{P}^n f^q)-\nu (f^qlog \nu f^q)+ \tilde C \nu\left\vert \nabla_{\Gamma_1} f
\right\vert^q+\tilde C \nu\left\vert \nabla_{\Gamma_0} f
\right\vert^q\\  &\label{2.6Pap2}
+\tilde C \sum_{k=1}^{n-1} \nu \left\vert \nabla_{\Gamma_0}( \mathcal{P}^kf^q)^\frac{1}{q}
\right\vert^q+\tilde C \sum_{k=0}^{n-1} \nu\left\vert \nabla_{\Gamma_1}(\mathbb{E}^{\Gamma_{0}} \mathcal{P}^kf^q)^\frac{1}{q}
\right\vert^q \end{align}
  In order to calculate the third and fourth term on the right-hand side of~\eqref{2.6Pap2} we will use the following proposition
\begin{proposition}\label{prp2.3Pap2}Assume  either  (H1) or (H2). If the conditions  (C1)-(C3) are satisfied, then the following bound holds
\begin{equation} \label{2.7Pap2}\nu\left\vert \nabla_{\Gamma_i}(\mathbb{E}^{\Gamma_{j}} \left\vert f\right\vert^q)^\frac{1}{q}
\right\vert^q\leq C_1\nu\left\vert \nabla_{\Gamma_i}f\right\vert^q+C_2\nu\left\vert \nabla_{\Gamma_j}f\right\vert^q\end{equation}
for $\{i,j\}=\{0,1\}$ and  constants $C_1\in (0,\infty)$ and $0<C_2<1$.\end{proposition}
The proof of Proposition~\ref{prp2.3Pap2} will be the subject of Section \ref{section6}.
 If we apply inductively the bound~\eqref{2.7Pap2}
  k times to  the third and the fourth term of~\eqref{2.6Pap2} we obtain
 \begin{equation}  \label{2.8Pap2}\nu\left\vert \nabla_{\Gamma_0}(\mathcal{P}^kf^q)^\frac{1}{q}
\right\vert^q \leq C_2^{2k-1}C_1\nu\left\vert \nabla_{\Gamma_1} f
\right\vert^q+C_2^{2k}\nu\left\vert \nabla_{\Gamma_0} f
\right\vert^q\end{equation}
     and
\begin{equation}\label{2.9Pap2}\nu\left\vert \nabla_{\Gamma_1}(\mathbb{E}^{\Gamma_{0}}\mathcal{P}^kf^q)^\frac{1}{q}
\right\vert^q \leq C_2^{2k}C_1\nu\left\vert \nabla_{\Gamma_1} f
\right\vert^q+C_2^{2k+1}\nu\left\vert \nabla_{\Gamma_0} f
\right\vert^q\end{equation}
  If we plug~\eqref{2.8Pap2} and~\eqref{2.9Pap2} in~\eqref{2.6Pap2}, we get 
\begin{align}\nonumber\nu (f^qlog\frac{f^q}{\nu f^q})\leq&\nu (\mathcal{P}^n f^q log \mathcal{P}^n f^q)-\nu (f^qlog \nu f^q)+ \tilde C \nu\left\vert \nabla_{\Gamma_1} f
\right\vert^q\\ \nonumber
 &+\tilde C(\sum_{k=0}^{n-1}C_2^{2k-1})C_1\nu\left\vert \nabla_{\Gamma_1} f
\right\vert^q+\tilde C(\sum_{k=0}^{n-1}C_2^{2k})\nu\left\vert \nabla_{\Gamma_0} f
\right\vert^q\\ &
 \label{2.10Pap2}+\tilde C(\sum_{k=0}^{n-1}C_2^{2k})C_1\nu\left\vert \nabla_{\Gamma_1} f
\right\vert^q+\tilde C(\sum_{k=0}^{n-1}C_2^{2k+1})\nu\left\vert \nabla_{\Gamma_0} f
\right\vert^q\end{align}
 If we take the limit of $n$  to infinity in~\eqref{2.10Pap2} the first two term on
 the right hand side cancel with each other,  as explained in the proposition
 bellow.
 \begin{proposition} \label{prp2.4Pap2} Assume  either  (H1) or (H2). If the conditions  (C1)-(C3) are satisfied, then $\mathcal{P}^nf$ converges
 $\nu$-almost everywhere to $\nu f$, where $\mathcal{P}$ as in (\ref{defPch6}).
 \end{proposition}
 The proof of this proposition will be presented in Section \ref{section5}.
So, taking the limit of $n$  to infinity in~\eqref{2.10Pap2} leads to  $$\nu (f^qlog\frac{f^q}{\nu f^q})\leq \left(\tilde C+cA\left(\frac{C_1}{C_2}+C_2+C_1\right)\right)\nu\left\vert \nabla_{\Gamma_1} f
\right\vert^q+\tilde CA\nu\left\vert \nabla_{\Gamma_0} f
\right\vert^q$$
 where $A=lim_{n\rightarrow\infty}\sum_{k=0}^{n-1}C_2^{2k}<\infty$  for  $C_2<1$, and the theorem follows for a constant $$\mathfrak{C}= max\{ \left(\tilde C+\tilde CA\left(\frac{C_1}{C_2}+C_2+C_1\right)\right),cA\}$$
\qed
\section{q Poincar\'e type Inequality.}
     In this section we present the proof of the q Spectral Gap type inequality (\ref{2.2Pap2}). In the case of quadratic interactions $V(x,y)=(x-y)^2$ one can calculate $$\mathbb{E}^{i,\omega}\left( f^2(\nabla_j V(x_i-x_j)-\mathbb{E}^{i,\omega}\nabla_j V(x_i-x_j))^2 \right)$$  (see \cite{B-H} and \cite{H}) with the use of the Deuschel-Stroock relative entropy inequality (see \cite{D-S}) and the Herbst argument (see \cite{L} and \cite{H}). Herbst's argument states that if a probability measure $\mu$ satisfies the (LS2) inequality and a function $F$ is Lipschitz continues with $\Vert F \Vert_{Lips}\leq 1$ and such that $\mu (F)=0$, then for some small $\epsilon$ we have  $$\mu e^{\epsilon F^2}<\infty$$
For $\mu=\mathbb{E}^{i,\omega}$ and $F=\frac{\nabla_j V(x_i-x_j)-\mathbb{E}^{i,\omega}\nabla_j V(x_i-x_j)}{2}$ we then obtain  $$\mathbb{E}^{i,\omega} e^{\frac{\epsilon}{4} (\nabla_j V(x_i-x_j)-\mathbb{E}^{i,\omega}\nabla_j V(x_i-x_j))^2}<\infty$$uniformly on the boundary conditions $\omega$, because of the hypothesis that  $\mathbb{E}^{i,\omega}$ satisfies
the Log-Sobolev-q Inequality with a constant $c$ uniformly with respect to the boundary
conditions $\omega$.
 In the more general case however examined in this work, where interactions may be  non quadratic and the LSq inequality   does not hold for $\mathbb{E}^{i,\omega} $ uniformly on $\omega$, the Herbst argument cannot be applied. In this and next sections, following \cite{Pa1}, we show how one can bound exponential quantities like the last one with the use of the projection of the infinite dimensional Gibbs measure and hypothesis (C1) and (C2). 

For every probability measure $\mu$, we define the correlation function $$\mu(f;g)\equiv\mu(fg)-\mu(f)\mu(g)$$For the function  $h_k:=f-\mathbb{E}^{\{ \sim k\}}f$ we define$$Q(u,k)\equiv\nu_{\Lambda(u)}\left\vert \nabla_{\Lambda(u)}\left( \mathbb{E}^{M(u)}
\vert h_k\vert ^q \right)^{\frac{1}{q}}\right\vert^q$$
  where the set $\Lambda(k)=\{k-2,k-1,k,k+1,k+2\}$ and  $M(k)=\mathbb{Z}\smallsetminus\Lambda(k)$.
 This quantity will be frequently used in the remaining section to bound the variance and the entropy. The following proposition presents a useful bound for $Q(k,k)$  under the hypothesis (C1), (C2) and (C3). The proof of this proposition can be found in \cite{Pa1}.    \begin{proposition}\label{[Pa]Q(k,k)} Suppose that  hypothesis (C1), (C2) and (C3)  are satisfied. Then \begin{align*} Q(k,k)\leq J^{q}S\nu \left \vert f- \mathbb{E}^{k-1}\mathbb{E}^{k+1}f \right \vert^q+S\sum_{r=k-2}^{k+2}\nu\left\vert \nabla_{r}
f
\right\vert^q+S\sum_{ r=3 }^{\infty} J_0^{r-2}\nu\left\vert \nabla_{k\pm r} f
\right\vert^q\end{align*}for some positive constant  $S$ and $J_0=J^\frac{q-1}{4}$.
\end{proposition}
The next lemma shows the Poincar\'e inequality for the two site  measure $\mathbb{E}^{\{\sim i\}},i\in \mathbb{Z}$ on the ball.
\begin{lemma}\label{PoincareBall} For any $L>0$ the following Poincar\'e inequality holds$$\mathbb{E}^{\{\sim i\}}\otimes\mathbb{\tilde E}^{\{\sim i\}}_{i}\vert f-\tilde f\vert^q\mathcal{I}_{\{\eta(i,\omega)+\tilde\eta(i,\omega)\leq L\}}\leqslant D_{L}\mathbb{E}^{\{\sim i\}}\vert \nabla_{\{\sim i \}} f\vert^q$$
where $\eta(i,\omega)=d( x_{i-1} )+d(x_{i+1} )+\sum_{j\sim\{\ i-1,i+1\}}d( \omega_{j})$ and $\mathcal{I}_A$ is the indicator function of set $A$.
\end{lemma}
\begin{proof} 
 \begin{align}\nonumber I_1:&= \mathbb{E}^{\{\sim i\}}\otimes\mathbb{\tilde E}^{\{\sim i\}}_{i}\vert f-\tilde f\vert^q\mathcal{I}_{\{\eta(i,\omega)+\tilde\eta(i,\omega)\leq L\}}\\  \nonumber&=\int\int\vert f-\tilde f\vert^q\mathbb{\mathcal{I}}_{\{\eta(i,\omega)+\tilde\eta(i,\omega)\leq L\}}\rho_{i}\hat\rho_{i}dX_{\{\sim i \}}d\hat X_{\{\sim i \}}
\\ & \nonumber\leq \int_{\{\eta(i,\omega)\leq L\}}\int_{\{\tilde\eta(i,\omega)\leq L\}}\vert f-\tilde f\vert^q\rho_{i}\hat\rho_i \mathbb{\mathcal{I}}_{\{\eta(i,\omega)\leq L\}}\mathbb{\mathcal{I}}_{\{\tilde\eta(i,\omega)\leq L\}}dX_{\{\sim i \}}d\hat X_{\{\sim i \}} \end{align}
where $\rho_{i}=\frac{e^{-H^{\{\sim i\},\omega}}}{\int e^{-H^{\{\sim i\},\omega}}dX_{\{\sim i \}} }$. Since   on $\{\eta(i,\omega)\leq L\}$ we have $d( x_{j} )\leq L, \ j=i-1,i+1$ and $\sum_{j\sim\{ i-1,i+1\}}d( \omega_{j})\leq L$, according to hypothesis (H$^*$) we can bound $\int e^{-H^{\{\sim i\},\omega}}dX_{\{\sim i \}}$ from bellow independently on the boundary conditions $\omega$. This leads to 
\begin{align}\label{1.11prop1.5p5}I_{1}\nonumber\leq&\frac{D^{*}(L)}{\int e^{-H^{\{\sim i\},\omega}}dX_{\{\sim i \}} }\times\ \\& \times \int_{\{\eta(i,\omega)\leq L\}}\int_{\{\tilde\eta(i,\omega)\leq L\}}\vert f-\tilde f\vert^q\mathbb{\mathcal{I}}_{\{\eta(i,\omega)\leq L\}}\mathbb{\mathcal{I}}_{\{\tilde\eta(i,\omega)\leq L\}}dX_{\{\sim i \}}d\hat X_{\{\sim i \}}
\end{align}for some positive constant $D^{*}(L)$. If we set  $B_R=\{x\in \mathbb{H}:d(x)\leq R\}$  then  (\ref{1.11prop1.5p5}) gives 
\begin{align}I_{1}&\leq\frac{D^{*}(L)}{\int e^{-H^{\{\sim i\},\omega}}dX_{\{\sim i \}} } \nonumber  \int_{B_{L}}\int_{B_{L}}\vert f\mathbb{\mathcal{I}}_{\{\eta(i,\omega)\leq L\}}-\tilde f\mathbb{\mathcal{I}}_{\{\tilde\eta(i,\omega)\leq L\}}\vert^qdX_{\{\sim i \}}d\hat X_{\{\sim i \}}\\ &\label{1.12prop1.5p5}\leq \frac{D^{*}(L)A_{L}}{\int e^{-H^{\{\sim i\},\omega}}dX_{\{\sim i \}} }  \int_{B_{L}}\vert\nabla_{\{\sim i \}} f\vert^q\mathbb{\mathcal{I}}_{\{\eta(i,\omega)\leq L\}} dX_{\{\sim i \}} \end{align}
where above we used the Poincar\'e Inequality in the Carnot-Caratheodory ball on the Heisenberg group with respect to the Haar measure (see \cite{V-SC-C}) with  constant $A_{L}$ depending only on the radius.  From (\ref{1.12prop1.5p5}) and hypothesis (H$^*$), we obtain  \begin{align*} I_{1}&\leq D_{*}(L)D^{*}(L)A_{L}  \int_{\{\eta(i,\omega)\leq L\}}\vert\nabla_{\{\sim i \}} f\vert^q  d\mathbb{E}^{\{\sim i\}}\\ & \leqslant D_{*}(L)D^{*}(L)A_{L}\mathbb{E}^{\{\sim i\}}\vert \nabla_{\{\sim i \}} f\vert^q
\end{align*}
And the lemma follows for appropriate constant $D_L$.
\end{proof} The following
 proposition gives a Spectral Gap type inequality for the measure  $\mathbb{E}^{\{i\},\omega}$.
\begin{proposition}\label{prp3.2Pap2}If conditions $(C1), (C2)$ and $(C3)$ are satisfied, then  the following                   
                     Spectral Gap type
inequality
$$\nu\mathbb{E}^{\{\sim i\}}\vert f-\mathbb{E}^{\{\sim i\}}f\vert^q\leq M\sum_{r=i-2}^{i+2}\nu\left\vert \nabla_{r}
f
\right\vert^q+M\sum_{ r=3 }^{\infty} J_0^{r-2}\nu\left\vert \nabla_{k\pm r} f
\right\vert^q$$
holds for a positive constant $M$.\end{proposition}
\begin{proof}
  If $\mathbb{\tilde E}^{\{\sim i\}}$ is an isomorphic copy of $\mathbb{E}^{\{\sim i\}}$ we can
then write
\begin{align}\nonumber\nu\vert f-\mathbb{E}^{\{\sim i\}}f\vert^q=&\nu\mathbb{E}^{\{\sim i\}}\vert f-\mathbb{E}^{\{\sim i\}}f\vert^q\leq\nu\mathbb{E}^{\{\sim i\}}\otimes\mathbb{\tilde E}^{\{\sim i\}}\vert f-\tilde f\vert^q\\   =&
\nu\mathbb{E}^{\{\sim i\}}\otimes\mathbb{\tilde E}^{\{\sim i\}}_{i}\vert f-\tilde f\vert^q\mathcal{I}_{\{\eta(i,\omega)+\tilde\eta(i,\omega)\leq L\}}\nonumber \\  &+\label{3.1Pap2}\nu\mathbb{E}^{\{\sim i\}}\otimes\mathbb{\tilde E}^{\{\sim i\}}\vert f-\tilde f\vert^q\mathbb{\mathcal{I}}_{\{\eta(i,\omega)+\tilde\eta(i,\omega)> L\}}\end{align}
where we have denoted $$\eta(i,\omega)=d( x_{i-1} )+d( x_{i+1} )+\sum_{j\sim\{\ i-1,i+1\}}d(\omega_{j} )$$ and $\mathcal{I}_A$ the indicator function of set $A$. For the first term on the right hand side of (\ref{3.1Pap2}) we can use Lemma \ref{PoincareBall} to obtain
\begin{equation*} \mathbb{E}^{\{\sim i\}}\otimes\mathbb{\tilde E}^{\{\sim i\}}\vert f-\tilde f\vert^q\mathbb{\mathcal{I}}_{\{\eta(i,\omega)+\tilde\eta(i,\omega)\leq L\}} \leq D_{L}\mathbb{E}^{\{\sim i\}}\otimes\mathbb{\tilde E}^{\{ \sim i\}}\left\vert \nabla_{\{\sim i\}}f\right\vert^q\end{equation*}
If we apply the Gibbs measure at the last inequality we obtain
 \begin{equation}\label{3.2Pap2}\nu\ \mathbb{E}^{\{\sim i\}}\otimes\mathbb{\tilde E}^{\{\sim i\}}\vert f-\tilde f\vert^q\mathbb{\mathcal{I}}_{\{\eta(i,\omega)+\tilde\eta(i,\omega)\leq L\}} \leq D_{L}\nu\left\vert \nabla_{\{\sim i\}}f\right\vert^q\end{equation}
 
  For the second term in~\eqref{3.1Pap2} we can write
\begin{align}\nonumber \nu\mathbb{E}^{\{\sim i\}}\otimes&\mathbb{\tilde E}^{\{\sim i\}}\vert f-\tilde f\vert^q\mathbb{\mathcal{I}}_{\{\eta(i,\omega)+\tilde\eta(i,\omega)> L\}}\leq\nu\mathbb{E}^{\{\sim i\}}\otimes\mathbb{\tilde E}^{\{\sim i\}}\vert f-\tilde f\vert^q\frac{\eta(i,\omega)+\tilde\eta(i,\omega)}{L}\\\leq
 &\frac{2^q2}{L}\nu\mathbb{E}^{\{\sim i\}}\left(\vert f-\mathbb{E}^{\{\sim i\}}f\vert^q\eta(i,\omega)\right)+\frac{2^q2}{L}\nu\mathbb{E}^{\{\sim i\}}\vert f-\mathbb{E}^{\{\sim i\}}f\vert^q\mathbb{E}^{\{\sim i\}}\eta(i,\omega)\nonumber
\\ =& \frac{2^q2}{L}\nu\left(\vert f-\mathbb{E}^{\{\sim i\}}f\vert^q\left(\eta(i,\omega)+\mathbb{E}^{\{\sim i\}}\eta(i,\omega)\right)\right)\nonumber \\ =& \label{3.3Pap2}\frac{2^q2}{L}\nu_{\Lambda(i)}\left[ \left(\mathbb{E}^{M(i)}
\vert f-\mathbb{E}^{\{\sim i\}}f\vert^q\right)\left(\eta(i,\omega)+\mathbb{E}^{\{\sim i\}}\eta(i,\omega)\right) \right]
\end{align}
where above we used that $\eta(i,\omega)+\mathbb{E}^{\{\sim i\}}\eta(i,\omega)$ is localised in $\Lambda(i)$ and that $M(i)=\mathbb{Z}\smallsetminus\Lambda(i)$. On the right hand side of~\eqref{3.3Pap2}  we can use the following Deuschel-Stroock entropic inequality
(see \cite{D-S})
\begin{equation}\label{3.4Pap2}\forall t>0, \ \mu(uv)\leq\frac{1}{t}log\left(\mu(e^{tu})\right)+\frac{1}{t}\mu(vlogv)\end{equation}
 for any   measure $\mu$  and $v\geq 0$ such that $\mu( v)=1$. Then from~\eqref{3.3Pap2}
 and~\eqref{3.4Pap2}  we will obtain
 \begin{align} \nonumber\nu\mathbb{E}^{\{\sim i\}}\otimes&\mathbb{\tilde E}^{\{\sim i\}}\vert f-\tilde f\vert^q\mathbb{\mathcal{I}}_{\{\eta(i,\omega)+\tilde\eta(i,\omega)> L\}}\\\leq
\nonumber
&  \frac{2^q2}{\epsilon L }\nu_{\Lambda(i)}\vert f-\mathbb{E}^{\{\sim i\}}f\vert^{q}log\frac{\vert f-\mathbb{E}^{\{\sim i\}}f\vert^{q}}{\nu_{\Lambda(i)}\vert f-\mathbb{E}^{\{\sim i\}}f\vert^{q}} \\ &+\label{3.5Pap2}\frac{2^q2}{\epsilon L}\left(log\nu_{\Lambda(i)}e^{\epsilon\left(\eta(i,\omega)+\mathbb{E}^{\{\sim i\}}\eta(i,\omega)\right)}\right)\nu_{\Lambda(i)}\mathbb{E}^{M(i)}\vert f-\mathbb{E}^{\{\sim i\}}f\vert^{q}\end{align}
The first term on the right hand side of~\eqref{3.5Pap2} can be bounded by the Log-Sobolev inequality for        
 $\nu_{\Lambda(i)}$  from hypothesis (C1)
   \begin{align} \nonumber \nu_{\Lambda(i)}\vert f-\mathbb{E}^{\{\sim i\}}
f\vert^{q}&log\frac{\vert f-\mathbb{E}^{\{\sim i\}}
f\vert^{q}}{\nu_{\Lambda(i)}\vert f-\mathbb{E}^{k-1}\mathbb{E}^{k+1}
f\vert^{q}}\\ &   \label{3.6Pap2}\leq C\nu_{\Lambda(i)}\left\vert \nabla_{\Lambda(i)}(\mathbb{E}^{M(i)}
\vert f-\mathbb{E}^{\{\sim i\}}f\vert^{q})^{\frac{1}{q}}\right \vert ^q=CQ(i,i)\end{align}
If we combine~\eqref{3.5Pap2} and~\eqref{3.6Pap2} together with hypothesis (C2) we get
  \begin{align} \nonumber\nu\mathbb{E}^{\{\sim i\}}\otimes\mathbb{\tilde E}^{\{\sim i\}}\vert f-\tilde f\vert^q&\mathcal{I}_{\{\eta(i,\omega)+\tilde\eta(i,\omega)> L\}}\\
& \label{3.7Pap2}\leq \frac{2^q2C}{\epsilon L }Q(i,i)+ \frac{2^q2K}{\epsilon L}\nu_{\Lambda(i)}\mathbb{E}^{M(i)}\vert f-\mathbb{E}^{\{\sim i\}}f\vert^{q}\end{align}
If we put together relationships ~\eqref{3.1Pap2}, ~\eqref{3.2Pap2} and~\eqref{3.7Pap2} we obtain
\begin{align}\nu\vert f-\mathbb{E}^{\{\sim i\}}f\vert^q\leq D_{L} B\nu\left\vert \nabla_{\{\sim i\}}f\right\vert^q+\frac{2^q2C}{\epsilon L }Q(i,i)
\label{3.8Pap2}+ \frac{2^q2K}{\epsilon L}\nu\vert f-\mathbb{E}^{\{\sim i\}}f\vert^{q}\end{align}
where the constant $K$ is as in (C2). If we use  the bound for $Q(i,i)$ from Proposition~\ref{[Pa]Q(k,k)}, then~\eqref{3.8Pap2} gives
\begin{align}\nonumber\nu\vert f-\mathbb{E}^{\{\sim i\}}f\vert^q\leq& D_{L} B\nu\left\vert \nabla_{\{\sim i\}}f\right\vert^q+\frac{2^q2CS}{\epsilon L }\sum_{r=i-2}^{i+2}\nu\left\vert \nabla_{r}
f
\right\vert^q\\  &+ \label{3.9Pap2}
\frac{2^q2CS}{\epsilon L }\sum_{ r=3 }^{\infty} J_0^{r-2}\nu\left\vert \nabla_{i\pm r} f
\right\vert^q+  \frac{2^q2K+J^{q}S2^q2C}{L\epsilon  }\nu\vert f-\mathbb{E}^{\{\sim i\}}f\vert^{q}\end{align}
For $L$ sufficiently large such that  $1-\frac{2^q2K+J^{q}S2^q2C}{L\epsilon  }>\frac{1}{2}$ we obtain
\begin{align*}\left( 1-\frac{2^q2K+J^{q}S2^q2C}{L\epsilon  } \right)\nu\vert f-\mathbb{E}^{ i}f\vert^q\leq &D_L\nu\left\vert \nabla_{\{\sim i\}}f\right\vert^q+
\frac{2^q2CS}{\epsilon L }\sum_{r=i-2}^{i+2}\nu\left\vert \nabla_{r}
f
\right\vert^q\\  &+\frac{2^q2S}{\epsilon L }\sum_{ r=3 }^{\infty} J_0^{r-2}\nu\left\vert \nabla_{i\pm r} f
\right\vert^q\end{align*}
 Which implies
$$\nu\vert f-\mathbb{E}^{\{\sim i\}}f\vert^q\leq M\sum_{r=i-2}^{i+2}\nu\left\vert \nabla_{r}
f
\right\vert^q+M\sum_{ r=3 }^{\infty} J_0^{r-2}\nu\left\vert \nabla_{i\pm r} f
\right\vert^q$$
 for some constant $M>0$.\end{proof}
If we combine together  Proposition~\ref{[Pa]Q(k,k)} and Proposition~\ref{prp3.2Pap2}, the following explicit bound for $Q(k,k)$
directly follows.  \begin{corollary}\label{crl3.3Pap2}Suppose that  hypothesis $(C1)$ and $(C2)$  are satisfied. Then \begin{align*} Q(k,k)\leq D\sum_{r=k-2}^{k+2}\nu\left\vert \nabla_{r}
f
\right\vert^q+D\sum_{ r=3 }^{\infty} J_0^{r-2}\nu\left\vert \nabla_{k\pm r} f
\right\vert^q\end{align*}where the constant $D=J^{q}SM+S$ and $J_0=J^\frac{q-1}{4}$.
\end{corollary}
\section{(LSq) type inequality under (H1).} Bellow an analogue result for the Log-Sobolev type inequality for  $\mathbb{E}^{ \{\sim i\},\omega}$ is presented assuming hypothesis $(H1)$.
\begin{proposition}\label{prp3.4Pap2} Assume  hypothesis   (H1). If the conditions  (C1)-(C3) are satisfied, then the following
Log-Sobolev  type inequality holds  
 $$\nu(\left\vert f\right\vert^qlog\frac{\left\vert f\right\vert^q}{\mathbb{E}^{ \{\sim i\},\omega}\left\vert f\right\vert^q}) \leq R_{1}\sum_{r=i-2}^{i+2}\nu\left\vert \nabla_{r}
f
\right\vert^q+R_{1}\sum_{ r=3 }^{\infty} J_0^{r-2}\nu\left\vert \nabla_{i\pm r} f
\right\vert^q$$for some positive constant $R_{1}$.\end{proposition}

  \begin{proof}
 Assume $f\geq 0$. We will use the Log-Sobolev Inequality for the $\mu$ measure to derive conditions for the Log-Sobolev inequality for the measure $\mathbb{E}^{\{\sim i\}}$. From hypothesis (H1) and Remark \ref{rem1.1} the product measure  $\mu(dx_{i+1})\otimes \mu(dx_{i-1})$ satisfies
the LSq with constant $c$.
\begin{align}\mu(dx_{i+1})\otimes \mu(dx_{i-1})(g^q\log\frac{g^q}{\mu(dx_{i+1})\otimes \mu(dx_{i-1}) g^q}&)\leq \label{3.10Pap2}c\mu(dx_{i+1})\otimes \mu(dx_{i-1})\left\vert \nabla g
\right\vert^q\end{align}
for $g\geq 0$. Define the function
$$h^i=-\sum_{j=\{i-1,i+1\},t\sim j}J_{j,t}V(x_{j},\omega_{t})$$
The function $h^i$ is localized   in $\Lambda(i)$. We also denote   $$\Phi^i=\phi(x_{i-1})+\phi(x_{i+1})$$
 Then inequality~\eqref{3.10Pap2} for  $g=e^\frac{{h^i}}{q} f,f\geq0$ gives
 \begin{align}\nonumber\int e^{-\Phi^i}&(e^{h^i} f^qlog\frac{e^{h^i} f^q}{\int e^{-\Phi^i}(e^{h^i} f^q)dx_{i-1}dx_{i+1}  \int e^{-\Phi^i}
dx_{i-1}dx_{i+1}})dx_{i-1}dx_{i+1} \\  &
\ \ \ \ \ \ \  \   \ \ \ \  \   \   \    \    \    \ 
 \    \    \     \   \    \ \    \    \    \    \    \    \   \    \    \
  \label{3.11Pap2}\leq c\sum_{j=i-1,i+1} \int e^{-\Phi^i} \left\vert \nabla_{j} (e^\frac{h^i}{q} f)
\right\vert^qdx_{i-1}dx_{i+1}\end{align}
 \noindent Denote by $I_r$ and $I_l$ the right and left hand side of~\eqref{3.11Pap2} respectively.
 If we use the Leibnitz rule
for the gradient on the right hand side of~\eqref{3.11Pap2} we have
 \begin{align}\nonumber      I_r\leq&2^{q-1}c\int  e^{-\Phi^i} \left\vert e^\frac{{h^i}}{q} ( \nabla_{j}  f)
\right\vert^q dx_{i-1}dx_{i+1}  \\  & \label{3.12Pap2}+2^{q-1} c\sum_{j=i-1,i+1} \int e^{-\Phi^i} \left\vert f (\nabla_{j} e^\frac{{h^i}}{q}  )
\right\vert^q dx_{i-1}dx_{i+1}= \nonumber\\ &  \left( \int e^{-\Phi^i+{h^i}}f^q
dx_{i-1}dx_{i+1} \right) c2^{q-1}\left( \mathbb{E}^{ \{\sim i\},\omega}\left\vert\nabla_{j} f
\right\vert^q+\frac{ 1}{q^{q}}\mathbb{E}^{ \{\sim i\},\omega}f^{q} \sum_{j=i-1,i+1}\left\vert  \nabla_{j}{h^i}
\right\vert^q\right)\end{align}
On the left hand side of~\eqref{3.11Pap2} we form the entropy for the measure $\mathbb{E}^{ \{\sim i\},\omega}$ measure with phase $\Phi^i- {h^i}$.

 \begin{align}\nonumber I_l = &\int e^{-\Phi^i+{h^i}} f^qlog\frac{f^q}{\int e^{-\Phi^i+{h^i}}f^q
dx_{i-1}dx_{i+1} / \int e^{-\Phi^i+{h^i}}
dx_{i-1}dx_{i+1}}dx_{i-1}dx_{i+1}\nonumber \\ &+\int e^{-\Phi^i+{h^i}} f^qlog\frac{\left(\int e^{-\Phi^i}
dx_{i-1}dx_{i+1}\right)  e^{h^i}}{ \int e^{-\Phi^i+{h^i}}
dx_{i-1}dx_{i+1}}dx_{i-1}dx_{i+1}\nonumber\\= &\nonumber\left( \int e^{-\Phi^i+{h^i}}f^q
dx_{i-1}dx_{i+1} \right) \left(\mathbb{E}^{ \{\sim i\},\omega}(f^qlog\frac{f^q}{ \mathbb{E}^{ \{\sim i\},\omega}f^q}) + \mathbb{E}^{ \{\sim i\},\omega}( f^q{h^i})\right)\nonumber \\ &+\int e^{-\Phi^i+{h^i}} f^qlog\frac{\int e^{-\Phi^i}
dx_{i-1}dx_{i+1}  e^{h^i}}{ \int e^{-\Phi^i+{h^i}}
dx_{i-1}dx_{i+1}}dx_{i-1}dx_{i+1} \end{align}
Since  ${h^i}$ is negative, because of hypothesis (C3), the last equality leads to 
 \begin{align}\label{3.13Pap2} I_l \geq\left( \int e^{-\Phi^i+{h^i}}f^q
dx_{i-1}dx_{i+1} \right)\left(\mathbb{E}^{ \{\sim i\},\omega}(f^qlog\frac{f^q}{ \mathbb{E}^{ \{\sim i\},\omega}f^q}) + \mathbb{E}^{ \{\sim i\},\omega}( f^q{h^i})\right) \end{align}
 Combining~\eqref{3.11Pap2} together with~\eqref{3.12Pap2} and~\eqref{3.13Pap2} we obtain
\begin{align}\nonumber\mathbb{E}^{ \{\sim i\},\omega}(f^qlog\frac{f^q}{\mathbb{E}^{ \{\sim i\},\omega}f^q}) \leq &2^{q-1}c\sum_{j=i-1,i+1}\mathbb{E}^{ \{\sim i\},\omega}\left\vert \nabla_{j} f
\right\vert^q\\ &+\mathbb{E}^{ \{\sim i\},\omega}\left(f^q(\frac{ c2^{q-1}\sum_{j=i-1,i+1}\left\vert  \nabla_{j}{h^i}
\right\vert^q}{q^{q}}-{h^i})\right) \nonumber \end{align}
If we apply the Gibbs measure in  the last relationship we have
\begin{equation}\label{3.14Pap2}\nu(f^qlog\frac{f^q}{\mathbb{E}^{ \{\sim i\},\omega}f^q}) \leq 2^{q-1}c\sum_{j=i-1,i+1}\nu\left\vert \nabla_{j} f
\right\vert^q+\nu(f^q(\frac{ c2^{q-1}\sum_{j=i-1,i+1}\left\vert  \nabla_j{h^i}
\right\vert^q}{q^{q}}-{h^i}))\end{equation}
 From  \cite{B-Z} and \cite{R}, for $1<q<2$ and $q=2$ respectively,  the following estimate
 of the entropy holds
 \begin{align}\label{Rothaus}\mathbb{E}^{ \{\sim i\},\omega}(\left\vert f\right\vert^qlog\frac{\left\vert f\right\vert^q}{\mathbb{E}^{ \{\sim i\},\omega}\left\vert f\right\vert^q}) \leq & \nonumber A\mathbb{E}^{ \{\sim i\},\omega}\left\vert f-\mathbb{E}^{ \{\sim i\},\omega}f \right\vert^{q} \\ &+
 \mathbb{E}^{ \{\sim i\},\omega}\left\vert f-\mathbb{E}^{ \{\sim i\},\omega}f \right\vert^qlog\frac{\left\vert f-\mathbb{E}^{ \{\sim i\},\omega}f \right\vert^q}{\mathbb{E}^{ \{\sim i\},\omega}\left\vert f-\mathbb{E}^{ \{\sim i\},\omega}f \right\vert^q}\end{align}
  for some positive constant $A$.  If we apply the Gibbs measure at the last inequality we get
 \begin{align}\nonumber\nu(\left\vert f\right\vert^qlog\frac{\left\vert f\right\vert^q}{\mathbb{E}^{ \{\sim i\},\omega}\left\vert f\right\vert^q}) \leq &A\nu\vert f-\mathbb{E}^{ \{\sim i\},\omega}f\vert^{q} \\ &
 \label{3.15Pap2}+\nu(\vert f-\mathbb{E}^{ \{\sim i\},\omega}f\vert^qlog\frac{\vert f-\mathbb{E}^{ \{\sim i\},\omega}f\vert^q}{\mathbb{E}^{ \{\sim i\},\omega}\vert f-\mathbb{E}^{ \{\sim i\},\omega}f\vert^q})\end{align}
 We can now use~\eqref{3.14Pap2} to bound the second term on the right hand side of~\eqref{3.15Pap2}. Then we will obtain
 \begin{align}\nonumber\nu(f^qlog &\frac{f^q}{\mathbb{E}^{ \{\sim i\},\omega}f^q}) \leq A\nu\left\vert f-\mathbb{E}^{ \{\sim i\},\omega}f \right\vert^{q} +2^{q-1}c\sum_{j=i-1,i+1}\nu\left\vert \nabla_{j} f
\right\vert^q\\ &
 +\nu\left(\vert f-\mathbb{E}^{ \{\sim i\},\omega}f\vert^{q} (\frac{ c2^{q-1}\sum_{j=i-1,i+1}\left\vert  \nabla_j{h^i}
\right\vert^q}{q^{q}}-{h^i})\right)\nonumber\\= &\nonumber A\nu\vert f-\mathbb{E}^{ \{\sim i\},\omega}f\vert^{q} +2^{q-1}c\sum_{j=i-1,i+1}\nu\left\vert \nabla_{j} f
\right\vert^q\\ &+
 \label{3.16Pap2}\nu_{\Lambda(i)}\left(\left(\mathbb{E}^{M(i)}\vert f-\mathbb{E}^{ \{\sim i\},\omega}f\vert^{q}\right)\left( \frac{ c2^{q-1}\sum_{j=i-1,i+1}\left\vert  \nabla_j{h^i}
\right\vert^q}{q^{q}}-{h^i}\right)\right)\end{align}
 where the last equality holds due to the fact that $h^i$ is localised in $\Lambda(i)$.
 We can bound the last term on the right hand side of~\eqref{3.16Pap2} with the use
 of the entropic inequality~\eqref{3.4Pap2} and the Log-Sobolev inequality for $\nu_{\Lambda(i)}$
 from (C1), in the same way we worked in Proposition~\ref{prp3.2Pap2}. Then we will get
 \begin{align}\nonumber\nu(f^qlog\frac{f^q}{\mathbb{E}^{ \{\sim i\},\omega}f^q}) \leq &\nonumber A\nu\vert f-\mathbb{E}^{ \{\sim i\},\omega}f\vert^{q} +2c\sum_{j=i-1,i+1}\nu\left\vert \nabla_{j} f
\right\vert^q\\ &+\frac{C}{\epsilon }\nu_{\Lambda(i)}\left\vert \nabla_{\Lambda(i)}\left( \mathbb{E}^{M(i)}
\left\vert f-\mathbb{E}^{i-1}\mathbb{E}^{i+1}f\right\vert^q \right)^{\frac{1}{q}}\right\vert^q
\nonumber\\ &+\frac{1}{\epsilon }\left(log\nu e^{\epsilon(\frac{ c2^{q-1}\sum_{j=i-1,i+1}\left\vert  \nabla_j{h^i}
\right\vert^q}{q^{q}}-{h^i})}\right)\nu\vert f-\mathbb{E}^{ \{\sim i\},\omega}f\vert^{q}\nonumber
\\ \leq&   \label{3.17Pap2}(A+\frac{K}{\epsilon})\nu\vert f-\mathbb{E}^{ \{\sim i\},\omega}f\vert^{q} +\frac{C}{\epsilon }Q(i,i)+2c\sum_{j=i-1,i+1}\nu\left\vert \nabla_{j} f
\right\vert^q
 \end{align}
 where at the last inequality we used  hypothesis (C2) to bound

 $\nu e^{\epsilon(\frac{ c2^{q-1}\sum_{j=i-1,i+1}\left\vert  \nabla_j{h^i}
\right\vert^q}{q^{q}}-{h^i})}$.  We can now use Corollary~\ref{crl3.3Pap2} to bound $Q(i,i)$
in~\eqref{3.17Pap2} as well as Proposition~\ref{prp3.2Pap2} to bound $\nu \vert f-\mathbb{E}^{ \{\sim i\},\omega}f \vert ^{q}$. We will then obtain
\begin{align}\nu(f^qlog\frac{f^q}{\mathbb{E}^{ \{\sim i\},\omega}f^q}) \nonumber\leq&(\frac{DC}{\epsilon }+AM+\frac{MK}{\epsilon})\sum_{ r=3 }^{\infty} J_0^{r-2}\nu\left\vert \nabla_{i\pm r} f
\right\vert^q\\ \nonumber&+(\frac{DC}{\epsilon }+AM+\frac{MK}{\epsilon})\sum_{r=i-2}^{i+2}\nu\left\vert \nabla_{r}
f
\right\vert^q+ 2c\sum_{j=i-1,i+1}\nu\left\vert \nabla_{j} f
\right\vert^q
  \end{align}
     The lemma follows for appropriate choice of the constant $R_{1}$.
\end{proof}
 \section{(LSq) type inequality under  (H2).}
 The proofs in this section follow closely mainly the methods used in \cite{H-Z}, but also in \cite{I-P}. We start with a proposition that shows how the non uniform U-bound and the Log-Sobolev inequality are related.
\begin{proposition}\label{prop1.2p5} Suppose that the measure $$d\mathbb{E}^{\{\sim i\},\omega}=\frac{e^{-H^{\{\sim i\},\omega}}dX_{\{\sim i \}}}{\int e^{-H^{\{\sim i\},\omega}}dX_{\{\sim i \}} }$$ 
satisfies the following non uniform U-bound\begin{equation}\label{Uboundprop1.2p5}\mathbb{E}^{\{\sim i\},\omega} \vert f\vert ^q\left(\vert \nabla_{\{\sim i\}}H^{\{\sim i\},\omega}\vert^q+H^{\{\sim i\},\omega}\right)\leq\hat C\mathbb{E}^{\{\sim i\},\omega} \vert \nabla_{i} f \vert^q+\hat D_{\{\sim i\}}(\omega)\mathbb{E}^{\{\sim i\},\omega}\vert f \vert ^q\end{equation}
for some  positive constant $\hat C$ and a function  $\hat D_{\{\sim i\}}(\omega)$ of $\omega$, both independent of $f$. Then the following defective Log-Sobolev inequality holds$$\mathbb{E}^{\{\sim i\},\omega} \left\vert f\right\vert^q log\frac{\left\vert f\right\vert^q}{\mathbb{E}^{\{\sim i\},\omega} \left\vert f\right\vert^q}\leq C \ \mathbb{E}^{\{\sim i\},\omega} \left\vert \nabla_{\{\sim i\}} f
\right\vert^q+C\mathbb{E}^{\{\sim i\},\omega}\vert f\vert^q+D_i(\omega)\mathbb{E}^{\{\sim i\},\omega}\vert f\vert ^q$$
where $C$ is a constant and $D_i(\omega)=\max\{\frac{2^{q-1}(q+\epsilon)\alpha}{\epsilon q^{q}},1\}\hat D_{\{\sim i\}}(\omega)$.
\end{proposition}
\begin{proof}
Without loss of generality we can assume that $f\geq 0$. We set $\rho_{i}=\frac{e^{-H^{\{\sim i\},\omega}}}{\int e^{-H^{\{\sim i\},\omega}}dX_{\{\sim i \}} }$ and $g=f\rho_{i} ^\frac{1}{q}$ We also assume that $$\int g^qdX_{\{\sim i \}}=\mathbb{E}^{\{\sim i\},\omega} f^q=1$$
Then we can write 
$$\int (g^q\log g^q)dX_{\{\sim i \}}=\frac{q}{\epsilon}\int g^q(\log g^\epsilon)dX_{\{\sim i \}}\leq \frac{q+\epsilon}{q}\frac{q}{\epsilon}\log\left(\int g^{q+\epsilon}dX_{\{\sim i \}}\right)^\frac{q}{q+\epsilon}$$
where above we used the Jensen's inequality. In order to bound the last expression we can use the   Classical Sobolev (C-S) inequality for the  Lebesgue measure $dX_{\{\sim i \}}$  (see \cite{V-SC-C}),
$$ \left(\int \vert f\vert ^{q+\epsilon} dX_{\{\sim i \}}\right)^\frac{q}{q+\epsilon}\leq\alpha \int\vert\nabla f  \vert^q dX_{\{\sim i \}} +\beta\int \vert f \vert ^q dX_{\{\sim i \}}  \   \   \   \   \   \   \   \   \   \       \text{(C-S)}$$ for positive constants $\alpha, \beta$. We will then obtain\begin{align}\nonumber\int (g^q\log g^q)dX_{\{\sim i \}}&\leq\frac{q+\epsilon}{\epsilon}\log\left(\alpha \int\vert\nabla_{\{\sim i\}} g \vert^qd X_{\{\sim i \}}+\beta\int \vert g \vert ^qdX_{\{\sim i \}}\right) \\ &\label{eq3prop1.2p5}\leq\frac{(q+\epsilon)\alpha}{\epsilon} \int\vert\nabla_{\{\sim i\}} g \vert^qd X_{\{\sim i \}}+\frac{(q+\epsilon)\beta}{\epsilon}\int \vert g\vert^q dX_{\{\sim i \}}\end{align}
where in the last inequality we used that $\log x\leq x$ for $x>0$.  For the first term on the right hand side of (\ref{eq3prop1.2p5}) we have
\begin{align}\nonumber\label{eq4prop1.2p5}\int \vert \nabla_{\{\sim i\}} g \vert^qdX_{\{\sim i \}}&=\int \vert \nabla_{\{\sim i\}} (f\rho^\frac{1}{q})  \vert^qdX_{\{\sim i \}}\\  &\leq2^{q-1}\mathbb{E}^{i,\omega} \vert \nabla_{\{\sim i\}} f  \vert^q+2^{q-1}\int \vert f\nabla_{\{\sim i\}} (\rho^\frac{1}{q})  \vert^qdX_{\{\sim i \}}\end{align}
We have \begin{align*}\int \vert f\nabla_{\{\sim i\}} (\rho^\frac{1}{q})  \vert^qdX_{\{\sim i \}}&=\int \vert \rho^\frac{1}{q}\rho^\frac{-1}{q}f\nabla_{\{\sim i\}} (\rho^\frac{1}{q})  \vert^qdX_{\{\sim i \}}=\mathbb{E}^{\{\sim i\},\omega} f^{q} \vert\rho^\frac{-1}{q} \nabla_{\{\sim i\}} (\rho^\frac{1}{q})  \vert^q\\ &=\frac{1}{q^{q}}\mathbb{E}^{\{\sim i\},\omega} f^{q} \vert\nabla_{\{\sim i\}} H^{\{\sim i\},\omega}  \vert^q\end{align*}
If we plug the last equality in (\ref{eq4prop1.2p5}), we  obtain 
\begin{equation}\label{eq5prop1.2p5} \int \vert \nabla_{\{\sim i\}} g \vert^qdX_{\{\sim i \}}\leq2^{q-1}\mathbb{E}^{\{\sim i\},\omega} \vert \nabla_{\{\sim i\}} f  \vert^q+\frac{2^{q-1}}{q^{q}}\mathbb{E}^{\{\sim i\},\omega}f^{q} \vert\nabla_{\{\sim i\}} H^{\{\sim i\},\omega}\vert^q
\end{equation}
If we combine inequalities  (\ref{eq3prop1.2p5}) and (\ref{eq5prop1.2p5}),  we get
\begin{align}\nonumber\int (g^q\log g^q)dX_{\{\sim i \}}\leq&\frac{2^{q-1}(q+\epsilon)\alpha}{\epsilon} \mathbb{E}^{\{\sim i\},\omega} \vert \nabla_{\{\sim i\}} f  \vert^q+\frac{(q+\epsilon)\beta}{\epsilon}\mathbb{E}^{\{\sim i\},\omega} f^q\\ \label{eq6prop1.2p5}&+\frac{2^{q-1}(q+\epsilon)\alpha}{\epsilon q^{q}}\mathbb{E}^{\{\sim i\},\omega} f^{q} \vert\nabla_{\{\sim i\}} H^{\{\sim i\},\omega}  \vert^q\end{align}
For the left hand side of (\ref{eq6prop1.2p5}), since $H^{\{\sim i\},\omega}\geq 0$,  we have 
\begin{align}\int (g^q\log g^q)dX_{\{\sim i \}}\nonumber=&\int \left(\frac{e^{-H^{\{\sim i\},\omega}}}{\int e^{-H^{\{\sim i\},\omega}}dX_{\{\sim i \}} }f^q\log \frac{e^{-H^{\{\sim i\},\omega}}}{\int e^{-H^{\{\sim i\},\omega}}dX_{\{\sim i \}} }f^q\right)dX_{\{\sim i \}}\\=&\mathbb{E}^{\{\sim i\},\omega} (f^q\log f^q)+\nonumber\mathbb{E}^{\{\sim i\},\omega} \left(f^q\log \frac{e^{-H^{\{\sim i\},\omega}}}{\int e^{-H^{\{\sim i\},\omega}}dX_{\{\sim i \}} }\right)
\\=&\nonumber\label{eq7prop1.2p5}\mathbb{E}^{\{\sim i\},\omega} (f^q\log f^q)-\mathbb{E}^{\{\sim i\},\omega} (f^qH^{\{\sim i\},\omega})\\& \nonumber-\mathbb{E}^{\{\sim i\},\omega} \left(f^q\log \int e^{-H^{\{\sim i\},\omega}}dX_{\{\sim i \}}\right)\\\geq&\mathbb{E}^{\{\sim i\},\omega} (f^q\log f^q)-\mathbb{E}^{\{\sim i\},\omega} (f^q H^{\{\sim i\},\omega})\end{align}
  If we combine (\ref{eq6prop1.2p5}) and (\ref{eq7prop1.2p5}), we obtain
\begin{align}\label{eq1.8prop1.2p5}\mathbb{E}^{\{\sim i\},\omega} (f^q\log f^q)\leq&\hat\alpha\mathbb{E}^{\{\sim i\},\omega} \vert \nabla_{\{\sim i\}} f  \vert^q+\hat \gamma\mathbb{E}^{\{\sim i\},\omega} f^q\nonumber\\ &+\hat \beta\mathbb{E}^{\{\sim i\},\omega} f^{q} \left(\vert\nabla_{\{\sim i\}} H^{\{\sim i\},\omega}  \vert^q+H^{\{\sim i\},\omega}\right)\end{align}
where $\hat \alpha=\frac{2^{q-1}(q+\epsilon)\alpha}{\epsilon}$, $\hat \beta=\max\{\frac{2^{q-1}(q+\epsilon)\alpha}{\epsilon q^{q}},1\}$ and $\hat \gamma =\frac{(q+\epsilon)\beta}{\epsilon}$. If we use the non uniform U-bound  (\ref{Uboundprop1.2p5}),  the inequality (\ref{eq1.8prop1.2p5}) gives
$$\mathbb{E}^{\{\sim i\},\omega} (f^q\log f^q)\leq(\hat\alpha+\hat \beta\hat C3)\mathbb{E}^{\{\sim i\},\omega} \vert \nabla_{\{\sim i\}} f  \vert^q+\hat \beta\hat D_{\{\sim i\}}(\omega)\mathbb{E}^{\{\sim i\},\omega} f^q+\hat \gamma\mathbb{E}^{\{\sim i\},\omega} f^q$$
If we replace $f$ with $\frac{f}{\mathbb{E}^{\{\sim i\},\omega}f }$ which has mean equal to one we obtain the result.
\end{proof} If we use Proposition \ref{prop1.2p5}  to bound the entropy term on the right hand side  of inequality (\ref{Rothaus}) we obtain the following corollary.
 
\begin{corollary}\label{col1.3p5} If condition (H2) is satisfied then the following inequality is true.
\begin{align}\nonumber\mathbb{E}^{ \{\sim i\},\omega}\left(\left\vert f\right\vert^qlog\frac{\left\vert f\right\vert^q}{\mathbb{E}^{ \{\sim i\},\omega}\left\vert f\right\vert^q}\right) \leq &C \ \mathbb{E}^{\{\sim i\},\omega} \left\vert \nabla_{\{\sim i\}} f
\right\vert^q+(A+C)\mathbb{E}^{ \{\sim i\},\omega}\left\vert f-\mathbb{E}^{ \{\sim i\},\omega}f \right\vert^{q} \\ & \label{1.9col1.3P5} +D_i(\omega)\mathbb{E}^{\{\sim i\},\omega}\vert f-\mathbb{E}^{ \{\sim i\},\omega}f\vert ^q\end{align}
\end{corollary}
Bellow we prove the Log-Sobolev type inequality (\ref{2.1Pap2}).

\begin{proposition}\label{lem1.7P5} Assume  hypothesis   (H2). If the conditions  (C1)-(C3) are satisfied, then the following
Log-Sobolev type inequality holds  
 $$\frac{1}{R_{2}}\nu(\left\vert f\right\vert^qlog\frac{\left\vert f\right\vert^q}{\mathbb{E}^{ \{\sim i\},\omega}\left\vert f\right\vert^q}) \leq \sum_{r=i-2}^{i+2}\nu\left\vert \nabla_{r}
f
\right\vert^q+\sum_{ r=3 }^{\infty} J_0^{r-2}\nu\left\vert \nabla_{i\pm r} f
\right\vert^q$$for some positive constant $R_{2}$ independent of $f$ and $i$.
\end{proposition}
\begin{proof} If we apply the Gibbs measure at the Log-Sobolev type inequality (\ref{1.9col1.3P5}), from Corollary \ref{col1.3p5} we have
\begin{align}\nonumber\nu(\left\vert f\right\vert^qlog\frac{\left\vert f\right\vert^q}{\mathbb{E}^{ \{\sim i\},\omega}\left\vert f\right\vert^q}) \leq &C \ \nu\left\vert \nabla_{\{\sim i\}} f
\right\vert^q+(A+C)\nu\left\vert f-\mathbb{E}^{ \{\sim i\},\omega}f \right\vert^{q} \\ &  \nonumber+\nu D_i(\omega)\vert f-\mathbb{E}^{ \{\sim i\},\omega}f\vert ^q\\  = & \nonumber C\nu\left\vert \nabla_{\{\sim i\}} f
\right\vert^q+(A+C)\nu\left\vert f-\mathbb{E}^{ \{\sim i\},\omega}f \right\vert^{q} \\ &  \label{1.22lem1.7P5} +\nu_{\Lambda(i)}D_i(\omega)\mathbb{E}^{M(i)}\vert f-\mathbb{E}^{ \{\sim i\},\omega}f\vert ^q\end{align}
where  we used that $D_i(\omega)$ is localised in $\Lambda(i)$ and that $M(i)=\mathbb{Z}\smallsetminus\Lambda(i)$. If we use again the entropic inequality (\ref{3.4Pap2}) as we did in Proposition \ref{prp3.2Pap2} to bound the last term on the right hand side of (\ref{1.22lem1.7P5}), we obtain 
\begin{align}\nonumber\nu(\left\vert f\right\vert^qlog\frac{\left\vert f\right\vert^q}{\mathbb{E}^{ \{\sim i\},\omega}\left\vert f\right\vert^q}) \leq & \nonumber C\nu\left\vert \nabla_{\{\sim i\}} f
\right\vert^q+(A+C)\nu\left\vert f-\mathbb{E}^{ \{\sim i\},\omega}f \right\vert^{q} \\ &+ \nonumber \frac{C}{\epsilon  }\nu_{\Lambda(i)}\mathbb{E}^{M(i)}\vert f-\mathbb{E}^{\{\sim i\}}f\vert^{q}log\frac{\mathbb{E}^{M(i)}\vert f-\mathbb{E}^{\{\sim i\}}f\vert^{q}}{\nu_{\Lambda(i)}\mathbb{E}^{M(i)}\vert f-\mathbb{E}^{\{\sim i\}}f\vert^{q}} \\ & +\frac{1}{\epsilon }\left(log\nu_{\Lambda(i)}e^{\epsilon\left(D_i(\omega)\right)}\right)\nu_{\Lambda(i)}\mathbb{E}^{M(i)}\vert f-\mathbb{E}^{\{\sim i\}}f\vert^{q} \nonumber\\ \leq&\nonumber C\nu\left\vert \nabla_{\{\sim i\}} f
\right\vert^q+(A+C+\frac{K}{\epsilon })\nu\left\vert f-\mathbb{E}^{ \{\sim i\},\omega}f \right\vert^{q} \\ & \label{1.23lem1.7p5} +\frac{C}{\epsilon  }Q(i,i) \end{align}
Where to obtain the last inequality we used the Log-Sobolev inequality for the measure $\nu_{\Lambda (i)}$ from hypothesis (C1).  If we use Proposition \ref{prp3.2Pap2} and Corollary  \ref{crl3.3Pap2} to bound the third and fourth term on the right hand side of (\ref{1.23lem1.7p5}), we finally get 

\begin{align*}\nonumber\nu(\left\vert f\right\vert^qlog\frac{\left\vert f\right\vert^q}{\mathbb{E}^{ \{\sim i\},\omega}\left\vert f\right\vert^q}) \leq & C\nu\left\vert \nabla_{\{\sim i\}} f
\right\vert^q+(A+C+\frac{K}{\epsilon })M\sum_{r=i-2}^{i+2}\nu\left\vert \nabla_{r}
f
\right\vert^q\\ &+(A+C+\frac{K}{\epsilon })M\sum_{ r=3 }^{\infty} J_0^{r-2}\nu\left\vert \nabla_{i\pm r} f
\right\vert^q \\ & + \frac{DC}{\epsilon  }\sum_{r=i-2}^{i+2}\nu\left\vert \nabla_{r}
f
\right\vert^q+\frac{DC}{\epsilon  }\sum_{ r=3 }^{\infty} J_0^{r-2}\nu\left\vert \nabla_{i\pm r} f
\right\vert^q \end{align*}
which gives the result for appropriate constant $R_{2}$ and $0<J_0<1$.
\end{proof}
 
\section{Proof of Proposition~\ref{prp2.2Pap2} and Proposition~\ref{prp2.3Pap2}.}\label{section6}
\
 We  first present some useful lemmata.  The first lemma provides a technical result for the correlation. The proof of the lemma can be found in \cite{Pa1}.\begin{lemma}\label{lem4.1Pap1} For any function $v_k$ localised in $\Lambda(k)$,    the following inequality holds  
\begin{align*}
\mathbb{E}^{\{\sim k\}}&(\vert f\vert^q;  v_{k})
 \leq c_{0}\left(\mathbb{E}^{\{\sim k\}}\vert f\vert^q\right)^{\frac{1}{p}}\left( \mathbb{E}^{\{\sim k \}}(\vert f-\mathbb{E}^{\{\sim k\}}f\vert^q\left(\vert v_k\vert^q+\mathbb{E}^{\{\sim k \}}\vert v_k\vert^{q}\right)) \right)^\frac{1}{q}\end{align*} for some constant $c_0$ uniformly on the boundary conditions. \end{lemma}
The
 next lemma presents an estimate  involving $Q(k,k)$.
 \begin{lemma} \label{lem4.2Pap1} Suppose that  hypothesis  (C1) is satisfied. Then
  \begin{align*}\nu\left(\mathbb{E}^{\{\sim k\}}  \vert f\vert^q\right)^{-\frac{q}{p}}\left\vert\mathbb{E}^{\{\sim k\}}(\vert f\vert^q;v_{k})\right\vert^q\leq&\frac{c^{q}_{0}C}{\epsilon  } Q(k,k)\\ &+\frac{c^{q}_{0}}{ \epsilon}\left( log\nu e^{\epsilon\left(\vert v_{k}\vert^{q}+\mathbb{E}^{\{\sim k\}}\vert v_k\vert^{q}\right)} \right)\nu\vert f-\mathbb{E}^{\{\sim k\}}f\vert^q\end{align*}
 for any function $v_{k}$ localised in $\Lambda(k)$.\end{lemma}
\begin{proof} We can start with the bound from Lemma~\ref{lem4.1Pap1}
  \begin{align}\nonumber \nu\left(\mathbb{E}^{\{\sim k\}} \vert f\vert^q\right)^{-\frac{q}{p}}\left\vert\mathbb{E}^{\{\sim k\}}(\vert f\vert^q;v_{k})\right\vert^q\leq
 & c^{q}_{0}\nu \mathbb{E}^{\{\sim k\}}(\vert f-\mathbb{E}^{\{\sim k\}}f\vert^q\left(\vert v_{k}\vert^{q}+\mathbb{E}^{\{\sim k\}}\vert v_k\vert^{q}\right))\\= &  \label{4.4Pap1}c^{q}_{0}\nu_{\Lambda(k)}\left((\mathbb{E}^{M(k)}\vert f-\mathbb{E}^{\{\sim k\}}
f\vert^q)\left(\vert v_{k}\vert^{q}+\mathbb{E}^{\{\sim k\}}\vert v_k\vert^{q}\right)\right) \end{align}
 because $\vert v_{k}\vert^{q}+\mathbb{E}^{\{\sim k\}}\vert v_k\vert^{q}$ is localised in $\Lambda(k)$. If we use the relative entropic inequality (\ref{3.4Pap2}) as we did in Proposition \ref{prp3.2Pap2}, together with hypothesis (C1)   we can bound~\eqref{4.4Pap1} by
\begin{align*}~\eqref{4.4Pap1}\leq &\frac{c^{q}_{0}C}{\epsilon  } Q(k,k)+ \frac{c^{q}_{0}}{ \epsilon}\left( log\nu e^{\epsilon\left(\vert v_{k}\vert^{q}+\mathbb{E}^{\{\sim k\}}\vert v_k\vert^{q}\right)} \right)\nu_{\Lambda(k)}\mathbb{E}^{\{\sim k\}}\vert f-\mathbb{E}^{\{\sim k\}}f\vert^q\end{align*}  \end{proof}
 Before we prove the sweeping out relations of Lemma  \ref{lem4.2Pap2} and Lemma \ref{lem3.5Pap2} we present two lemmata whose proof can be found in \cite{Pa1}. 
\begin{lemma} \label{lem4.3Pap1} The following inequality is satisfied
 \begin{align*}\nu\left\vert \nabla_{j}(\mathbb{E}^{\{\sim i \}}\left\vert f\right\vert^q)^\frac{1}{q}
\right\vert^q \leq& c_1\nu\left\vert \nabla_j f
\right\vert^q\\  &+\frac{J^qc_1}{q^q}  \nu\left(\mathbb{E}^{\{\sim i\}}  \left\vert f\right\vert^q\right)^{-\frac{q}{p}}\left\vert\mathbb{E}^{\{\sim i\}} (\left\vert f\right\vert^q;\sum_{t\in\{i-2,i,i+2\}:t\sim
j
}\nabla_j
V(x_t,\omega_j))\right\vert^q \end{align*}
for $j=i-2,i,i+2$.
 \end{lemma}
 \begin{lemma}\label{lem3.1Pap1}
  Under hypothesis  (C1), for any functions  $u$ localised in $\Lambda(k)$  the following inequality is satisfied
$$
\nu \left\vert\mathbb{E}^{k-1}\mathbb{E}^{k+1}(f;u)\right\vert^q\leq\frac{C}{\epsilon }Q(k,k)+\frac{1}{\epsilon }\left(log\nu_{\Lambda(k)} e^{\epsilon \vert u-\mathbb{E}^{k-1}\mathbb{E}^{k+1}u\vert^q}\right)\nu\vert f-\mathbb{E}^{k-1}\mathbb{E}^{k+1}f\vert^q
$$ for $\epsilon>0$.
\end{lemma} 
  \begin{lemma} \label{lem4.2Pap2}Assume  either   (H1) or (H2). If the conditions  (C1)-(C3) are satisfied, then the following
Log-Sobolev . Then
$$\nu\left\vert \nabla_{\Gamma_i}(\mathbb{E}^{\Gamma_j}f)
\right\vert^q \leq D_1\mathcal{\nu}\left\vert \nabla_{\Gamma_i} f
\right\vert^q+D_2\nu\left\vert \nabla_{\Gamma_j} f
\right\vert^q$$ for $\{i,j\}=\{0,1\}$ and  constants $D_{1}>0$ and  $0<D_{2}<1$.  \end{lemma}

 \begin{proof} Without loss of generality assume $i=0$ and $j=1$. We have
 \begin{equation}\label{4.5Pap2}\nu\left\vert \nabla_{\Gamma_1}(\mathbb{E}^{\Gamma_{0}}f)
\right\vert^q=\sum_{i\in \Gamma_1} \nu\left\vert \nabla_{i}(\mathbb{E}^{\Gamma_{0}}f)
\right\vert^q\leqslant\sum_{i\in \Gamma_1} \nu\left\vert \nabla_{i}(\mathbb{E}^{i-1}\mathbb{E}^{i+1}f)
\right\vert^q\end{equation}
 If we denote   $\rho_i= \frac{e^{-H(x_{i-1})}e^{-H(x_{i+1})}}{\int e^{-H(x_{i-1})}dx_i\int e^{-H(x_{i+1})}dx_i}$ the density of the measure $\mathbb{E}^{i-1}\mathbb{E}^{i+1}$ we can then write
 \begin{align}\nonumber\nu&\left\vert\nabla_{i}(\mathbb{E}^{i-1}\mathbb{E}^{i+1}f)\right\vert^{q}=
 \nu\left\vert\nabla_{i}(\int \int \rho_{i}  f dx_{i-1}dx_{i+1})\right\vert^{q}\leq
   &\\ &  \nonumber2^{q-1}\nu\left\vert\int \int (\nabla_{i}f) \rho_{i} dx_{i-1}dx_{i+1}\right\vert^{q}+\ 2^{q-1}\nu\left\vert\int \int f(\nabla_{i}\rho_{i} )dx_{i-1}dx_{i+1}\right\vert^{q}\leq
\\ & \label{3.9Pap1}c_{1}\nu\left\vert\mathbb{E}^{i-1}\mathbb{E}^{i+1}(\nabla_{i}f)\right\vert^{q}+\ c_{1}J^q\nu\left\vert\mathbb{E}^{i-1}\mathbb{E}^{i+1}(f; \nabla_{i}V(x_{i-1},x_{i})+\nabla_{i}V(x_{i+1},x_{i}))\right\vert^{q}
\end{align}
where in~\eqref{3.9Pap1} we used hypothesis (C3) to bound the coefficients $J_{i,j}$
and
 we have set $c_1=2^{4q}$. If we apply the H\"older  Inequality to the first term of~\eqref{3.9Pap1} and Lemma~\ref{lem3.1Pap1}  to the
second term, we obtain 
\begin{align}\nu\left\vert\nabla_{i}(\mathbb{E}^{i-1}\mathbb{E}^{i+1}f)\right\vert^q
\leq &c_{1}\nu\left\vert\nabla_{i}f\right\vert^q+\frac{c_{1}J^qC}{\epsilon }Q(i,i)+ \label{4.8Pap2}\frac{c_{1}KJ^q}{\epsilon }\nu\left\vert f-\mathbb{E}^{\{\sim i\}}f \right\vert^{q} \end{align}
where the constant $K$ as in hypothesis $(C2)$. 

If we use Corollary~\ref{crl3.3Pap2} to bound $Q(i,i)$ and Proposition~\ref{prp3.2Pap2} to bound the
last term on the right hand side of \eqref{4.8Pap2}  we obtain\begin{align}\nonumber\nu\left\vert\nabla_i(\mathbb{E}^{i-1}\mathbb{E}^{i+1}f)\right\vert^q
\leq &c_1\nu\left\vert\nabla_i f\right\vert^q+\frac{c_1J^qDC}{\epsilon }\sum_{r=i-2}^{i+2}\nu\left\vert \nabla_r
f
\right\vert^q+\frac{c_1J^qDC}{\epsilon }\sum_{ r=3 }^{\infty} J_0^{r-2}\nu\left\vert \nabla_{i\pm r} f
\right\vert^q \\ &\label{4.9Pap2}+\frac{c_1KMJ^q}{\epsilon }\sum_{r=i-2}^{i+2}\nu\left\vert \nabla_r
f
\right\vert^q+\frac{c_1KMJ^q}{\epsilon }\sum_{r=3}^{\infty} J_0^{r-2}\nu\left\vert \nabla_{i\pm r} f
\right\vert^q \end{align}
for the constants $D$ as in Corollary~\ref{crl3.3Pap2} and $M$  as in  Proposition~\ref{prp3.2Pap2}. From~\eqref{4.5Pap2} and~\eqref{4.9Pap2} we have
\begin{align*}
\nu  \left\vert \nabla_{\Gamma_1}(\mathbb{E}^{\Gamma_0}f)
\right\vert^q\leq & c_1\nu\left\vert\nabla_{ \Gamma_1}f\right\vert^q+\frac{\left(MK+D\right)c_{1}J^q}{\epsilon}\sum_{i\in \Gamma_1}\sum_{r=i-2}^{i+2}\nu\left\vert \nabla_r
f
\right\vert^q
 \\ &+\frac{\left(MK+D\right)c_{1}J^q}{\epsilon}\sum_{i\in \Gamma_1}\sum_{ r=3 }^{\infty} J_0^{r-2}\nu\left\vert \nabla_{i\pm r} f
\right\vert^q\end{align*}
 The last one implies
\begin{align*}\nonumber\nu  \left\vert \nabla_{\Gamma_1}(\mathbb{E}^{\Gamma_0}f)\right\vert^q\leq &   c_1\nu\left\vert\nabla_{ \Gamma_1}f\right\vert^q+\frac{\left(MK+D\right)c_{1}J^q}{\epsilon}(3+\sum_{n=0 }^{\infty} J_0^{2n})\nu\left\vert\nabla_{ \Gamma_1}f\right\vert^q\\  &+\frac{\left(MK+D\right)c_{1}J^q}{\epsilon}(2+\sum_{n=1 }^{\infty} J_0^{2n-1})\nu\left\vert\nabla_{ \Gamma_0}f\right\vert^q \end{align*}
If we choose $J$ in (C3) sufficiently small such that $J_0<1$ we finally obtain
 \begin{align*}\nu  \left\vert \nabla_{\Gamma_1}(\mathbb{E}^{\Gamma_0}f)
\right\vert^q\leq &J^{q}\frac{\left(MK+D\right)c_{1}J^q}{\epsilon}(2+\frac{J_0}{1-J_0^{2}})\nu\left\vert \nabla_ { \Gamma_0}
f
\right\vert^q\\&
  + \left(  c_1+\frac{\left(MK+D\right)c_{1}J^q}{\epsilon}\left(3+\frac{1}{1-J_0^{2}}\right)\right)\nu\left\vert\nabla_{ \Gamma_1}f\right\vert^q\end{align*}
and the lemma follows for $$D_1=c_1+\frac{\left(MK+D\right)c_1J^q}{\epsilon}\left(3+\frac{1}{1-J_0^2}\right)$$ and $J$ sufficiently small such that
$$D_2= J^q\frac{\left(MK+D\right)c_1J^q}{\epsilon}(2+\frac{J_0}{1-J_0^4})<1$$
\end{proof}
We continue with another sweeping out property which will play the basis of the proof of Proposition~\ref{prp2.3Pap2}.  \begin{lemma}
\label{lem3.5Pap2}Assume  either   (H1) or (H2). If the conditions  (C1)-(C3) are satisfied, then \begin{align*} \nu\left\vert \nabla_{j}
\left(\mathbb{E}^{ \{\sim i\}}\left\vert f\right\vert^q\right)^\frac{1}{q}
\right\vert^q\leq H\nu\left\vert \nabla_j f
\right\vert^q+HJ^q\sum_{ r=2 }^{\infty} J_0^{r-2}\nu\left\vert \nabla_{j\pm r} f
\right\vert^q\end{align*}for $j=i-2,i,i+2$ and some positive constant $H$.
\end{lemma} 
\begin{proof}  Assume $f\geq 0$. For $j=i-2,i,i+2$, from   Lemma~\ref{lem4.3Pap1}  we have
 \begin{align*}\nu\left\vert \nabla_{j}(\mathbb{E}^{\{\sim i \}}f^q)^\frac{1}{q}
\right\vert^q \leq& c_1\nu\left\vert \nabla_j f
\right\vert^q+\\  &\frac{J^qc_1}{q^q}  \nu\left(\mathbb{E}^{\{\sim i\}}  f^q\right)^{-\frac{q}{p}}\left\vert\mathbb{E}^{\{\sim i\}} (f^q;\sum_{t\in\{i-2,i,i+2\}:t\sim
j
}\nabla_j
V(x_t,\omega_j))\right\vert^q \end{align*}
If we bound the last term from Lemma~\ref{lem4.2Pap1} we obtain
 \begin{align}\nu\left\vert \nabla_{j}(\mathbb{E}^{\{\sim i \}}f^q)^\frac{1}{q}
\right\vert^q \leq& \nonumber c_1\nu\left\vert \nabla_j f
\right\vert^q+\frac{J^qc^{q}_{0}Cc_1}{\epsilon q^q}   Q(i,i)\\  & +\label{5.9Pap2}\frac{J^qc_1c^{q}_{0}}{ \epsilon q^q}\left( log\nu_{\Lambda(i)}e^{\epsilon\left(\vert W_{i}\vert^q+\mathbb{E}^{\{\sim k\}}\vert W_i\vert^{q}\right)} \right)\nu \vert f-\mathbb{E}^{\{\sim i\}} f\vert^q \end{align}
where above we denoted $ W_i=\sum_{t\in\{i-2,i,i+2\}:t\sim
j
}\nabla_j
V(x_t,\omega_j)$. We can  make use of hypothesis (C2) to bound$$log\nu_{\Lambda(i)}e^{\epsilon\left(\vert W_{i}\vert^q+\mathbb{E}^{\{\sim k\}}\vert W_i\vert^{q}\right)}<K$$as well as use Proposition~\ref{prp3.2Pap2} and  Corollary~\ref{crl3.3Pap2}  to bound   $\nu\vert f-\mathbb{E}^{\{\sim i\}} f \vert^q$ and
$Q(i,i)$. Then~\eqref{5.9Pap2} becomes
 \begin{align}\nu\left\vert \nabla_{j}(\mathbb{E}^{\{\sim i \}}f^q)^\frac{1}{q}
\right\vert^q \leq& \nonumber c_1\nu\left\vert \nabla_j f
\right\vert^q+\frac{J^qc^{q}_{0}C(M+D)c_1}{\epsilon q^q}   \sum_{r=i-2}^{i+2}\nu\left\vert \nabla_{r}
f
\right\vert^q\\  \label{5.10Pap2} &+\frac{J^qc^{q}_{0}C(M+D)c_1}{\epsilon q^q}\sum_{ r=3 }^{\infty} J_0^{r-2}\nu\left\vert \nabla_{i\pm r} f
\right\vert^q \end{align}
The proof of Lemma~\ref{lem3.5Pap2} is complete for appropriate choice of constant $H$. \end{proof}
  We will finish this section with the Proposition~\ref{prp2.3Pap2}.

 ~

\noindent\textit{\textbf{Proof of Proposition~\ref{prp2.3Pap2}. }} Assume $f\geq 0$. We can write
\begin{align}  \label{5.11Pap2}\nu\left\vert \nabla_{\Gamma_1}(\mathbb{E}^{\Gamma_{0}} f^q)^\frac{1}{q}
\right\vert^q=&\sum_{i\in \Gamma_1} \nu\left\vert \nabla_{i}(\mathbb{E}^{\Gamma_{0}} f^q)^\frac{1}{q}
\right\vert^q\leq\sum_{i\in \Gamma_1}\nu\left\vert \nabla_{i}(\mathbb{E}^{\{\sim i \}}f^q)^\frac{1}{q}
\right\vert^q\end{align}
 If we substitute in~\eqref{5.11Pap2} the bound from Lemma~\ref{lem3.5Pap2}, we obtain
 \begin{align*} \nu\left\vert \nabla_{\Gamma_1}(\mathbb{E}^{\Gamma_{0}} f^q)^\frac{1}{q}
\right\vert^q\leq & H\sum_{i\in \Gamma_1}\nu\left\vert \nabla_{i}
f
\right\vert^q+J^qH\sum_{i\in \Gamma_1}\sum_{ r=2 }^{\infty} J_0^{r-2}\nu\left\vert \nabla_{i\pm r} f
\right\vert^q\\ =  &
  H\nu\left\vert \nabla_{ \Gamma_1} f
\right\vert^q+J^qH\left(\sum_{ r=0 }^{\infty} J^R_0\right)(\nu\left\vert \nabla_{ \Gamma_0} f
\right\vert^q+\nu\left\vert \nabla_{ \Gamma_1} f
\right\vert^q)\\= & \nonumber (H+\frac{J^qH}{1-J_0})\nu\left\vert \nabla_{ \Gamma_1} f
\right\vert^q+\frac{J^qH}{1-J_0}\nu\left\vert \nabla_{ \Gamma_0} f
\right\vert^q\end{align*}
 For $J$ in (C3) sufficiently small  such that $\frac{J^qH}{1-J_0}<1$,
the proposition follows for constants
 $$C_{1}=H+\frac{J^qH}{1-J_0} \text{\; and \;} C_{2}=\frac{J^qH}{1-J_0}<1$$\qed

We can now prove Proposition~\ref{prp2.2Pap2}. 
 
~

\noindent
 \textit{\textbf{Proof of  Proposition~\ref{prp2.2Pap2}.}}   We will prove Proposition~\ref{prp2.2Pap2} for $k=1$,
 that is  $$\nu\mathbb{E}^{ \Gamma_{1}}(f^qlog\frac{f^q}{\mathbb{E}^{ \Gamma_{1}}f^q}) \leq\tilde C \nu\left\vert \nabla_{\Gamma_0}
f
\right\vert^q+\tilde C\nu\left\vert \nabla_{\Gamma_1} f
\right\vert^q$$ for $f\geq 0$. Define the sets $$a_0\equiv\{\sim 0\}\text{ \;and \;}   a_k\equiv\begin{cases}\{\sim(2 k+2)\} & \text{ \;for
\;} k\in \mathbb{N}\ \text{\; odd  \;} \\
\{\sim(-2k)\} & \text{ \;for
\;} k\in \mathbb{N}\ \text{\; even  \;} \\
\end{cases}$$ and consider the following representation
 of the odd integers $\Gamma_1$ \begin{align*}\Gamma_1=\cup_{k=0}^{+\infty} a_k=\{\sim0\}\cup\{\sim4\}\cup\{\sim(-4)\}\cup\{\sim8\}\cup\{\sim(-8)\}\cup...\end{align*}where
 we have denoted $\{\sim k\}=\{j\in \mathbb{Z} : j\sim k\}=\{k-1,k+1\}$. Then
  we can write
\begin{align}\nonumber\nu\mathbb{E}^{ \Gamma_{1}}(f^qlog\frac{f^q}{\mathbb{E}^{ \Gamma_{1}}f^q}) =&\nu\mathbb{E}^{ \{\sim0\}}(f^qlog\frac{f^q}{\mathbb{E}^{ \{\sim0\}}f^q})\\   &\label{3.18Pap2}+\sum_{k=1}^{+\infty}\nu\mathbb{E}^{ a_k}(\mathbb{E}^{ a_{k-1}}...\mathbb{E}^{ a_0}f^q\log\frac{\mathbb{E}^{ a_{k-1}}...\mathbb{E}^{ a_0}f^q}{\mathbb{E}^{ a_k}...\mathbb{E}^{ a_0}f^q})\end{align}
If we use Proposition~\ref{prp3.4Pap2} and Proposition \ref{lem1.7P5} in the case of hypothesis $\mathbf{A}$ and   $\mathbf{B}$ respectively to bound the first term in~\eqref{3.18Pap2} we have   \begin{align}\label{3.19Pap2}\nu\mathbb{E}^{ \{\sim0\}}(f^qlog\frac{f^q}{\mathbb{E}^{ \{\sim0\}}f^q})\leq R\sum_{r=-2}^{2}\nu\left\vert \nabla_{r}
f
\right\vert^q+R\sum_{ r=3 }^{\infty} J_0^{r-2}\nu\left\vert \nabla_{\pm r} f
\right\vert^q \end{align}  where $R=\max\{R_1,R_2\}$ for the constants $R_1$ and $R_2$ as in Proposition~\ref{prp3.4Pap2} and Proposition \ref{lem1.7P5}. For the terms in the sum in the last term of~\eqref{3.18Pap2} , for $k$ odd we have \begin{align}\nonumber
\nu\mathbb{E}^{ a_k}(\mathbb{E}^{ a_{k-1}}...\mathbb{E}^{ a_0}f^q\log\frac{\mathbb{E}^{ a_{k-1}}...\mathbb{E}^{ a_0}f^q}{\mathbb{E}^{ a_k}...\mathbb{E}^{ a_0}f^q})\leq &R\sum_{r=2k}^{2k+4}\nu\left\vert \nabla_{r}
(\mathbb{E}^{ a_{k-1}}...\mathbb{E}^{ a_0}f^q)^\frac{1}{q}
\right\vert^q\\   &\label{3.20Pap2}+R\sum_{ r=3 }^{\infty} J_0^{r-2}\nu\left\vert \nabla_{2k+2\pm r} (\mathbb{E}^{ a_{k-1}}...\mathbb{E}^{ a_0}f^q)^\frac{1}{q}\right\vert^q\end{align}  while for $k$ even we have \begin{align}\nonumber
\nu\mathbb{E}^{ a_k}(\mathbb{E}^{ a_{k-1}}...\mathbb{E}^{ a_0}f^q\log\frac{\mathbb{E}^{ a_{k-1}}...\mathbb{E}^{ a_0}f^q}{\mathbb{E}^{ a_k}...\mathbb{E}^{ a_0}f^q})\leq &R\sum_{r=-2k-2}^{-2k+2}\nu\left\vert \nabla_{r}
(\mathbb{E}^{ a_{k-1}}...\mathbb{E}^{ a_0}f^q)^\frac{1}{q}
\right\vert^q\\   &\label{3.22Pap2}+R\sum_{ r=3 }^{\infty} J_0^{r-2}\nu\left\vert \nabla_{-2k\pm r} (\mathbb{E}^{ a_{k-1}}...\mathbb{E}^{ a_0}f^q)^\frac{1}{q}\right\vert^q\end{align} For the quantities involved in~\eqref{3.20Pap2} and~\eqref{3.22Pap2}, if we define  $m_k=\min\{i:i\in\cup_{j=0}^{k-1} a_j\}$ and $M_k=\max\{i:i\in\cup_{j=0}^{k-1} a_j\}$, we then have 
\begin{align}\nu\left\vert \nabla_{s}
(\mathbb{E}^{ a_{k-1}}...\mathbb{E}^{ a_0}f^q)^\frac{1}{q}
\right\vert^q\label{3.21Pap2}\leq\begin{cases}\nu\left\vert \nabla_{s}
(\mathbb{E}^{ \{\sim s\}}f^q)^\frac{1}{q}
\right\vert^q   & \text{if \;} s \text{\; is even, }s\in(m_{k},M_{k})
\\
\nu\left\vert \nabla_{s}
(\mathbb{E}^{ \{\sim m_{k}\}}f^q)^\frac{1}{q}
\right\vert^q &   \text{if \;}s= m_{k}-1  \\
\nu\left\vert \nabla_{s}
(\mathbb{E}^{ \{\sim M_{k}\}}f^q)^\frac{1}{q}
\right\vert^q   & \text{if \;}s=M_{k}+1  \\
\ \ \ \ \ \ \ \ \ \ \ 0 & \text{if \;} s \text{\; is odd, }s\in[ m_{k},M_{k}]\\ \nu\left\vert \nabla_{s}
f\right\vert^q  & \text{if \;} s\notin [m_k,M_k] 
\end{cases}
\end{align}
 From relationships~\eqref{3.18Pap2} -~\eqref{3.21Pap2} we derive that the right hand side of~\eqref{3.18Pap2}
is reduced to an infinite sum of the following terms 
$$J_0^r\nu\left\vert \nabla_{s}
(\mathbb{E}^{ \{\sim s\}}f^q)^\frac{1}{q}
\right\vert^q \text{\; and \;} J_0^r\nu\left\vert \nabla_{t}
(\mathbb{E}^{ \{\sim s\}}f^q)^\frac{1}{q}
\right\vert^q \text{\; for \;}t=\{s-2,s+2\}$$
for every $s\in \Gamma_0$ and $r \in \mathbb{N}$.  For every $s$ and
$t$ the  above terms  are repeated at most two times for every different  $r$. So,
$\nu\left\vert \nabla_{s}
(\mathbb{E}^{ \{\sim s\}}f^q)^\frac{1}{q}
\right\vert^q \text{\; and \;} \nu\left\vert \nabla_{t}
(\mathbb{E}^{ \{\sim s\}}f^q)^\frac{1}{q}
\right\vert^q $ occur in the sum at most $2\sum_{n=0}^{+\infty}J_0^n$
 times each. Thus, we finally obtain\begin{align}\nonumber\nu\mathbb{E}^{ \Gamma_{1}}(f^q&\log\frac{f^q}{\mathbb{E}^{ \Gamma_{1}}f^q})\leq  R\sum_{r=-2}^{2}\nu\left\vert \nabla_{r}
f
\right\vert^q+R\sum_{ r=3 }^{\infty} J_0^{r-2}\nu\left\vert \nabla_{\pm r} f
\right\vert^q\\   &+2(  \sum_{n=0}^{+\infty}J_0^n)R\sum_{s\in\mathbb{Z}}\nu\left\vert \nabla_{2s}
(\mathbb{E}^{ \{\sim2 s\}}f^q)^\frac{1}{q}
\right\vert^q\nonumber \\ & \nonumber+2( \sum_{n=0}^{+\infty}J_0^n)R\sum_{s\in \mathbb{N}}\left(\nu\left\vert \nabla_{2s+2}
(\mathbb{E}^{ \{\sim2 s\}}f^q)^\frac{1}{q}
\right\vert^q+\nu\left\vert \nabla_{2s-2}
(\mathbb{E}^{ \{\sim(-2 s)\}}f^q)^\frac{1}{q}
\right\vert^q\right)\end{align}
If we choose  $J$ in (C3) sufficient small such that $J_0=J^\frac{q-1}{4}<1$, the last leads to  
\begin{align}\nonumber \nu\mathbb{E}^{ \Gamma_{1}}(f^q\log\frac{f^q}{\mathbb{E}^{ \Gamma_{1}}f^q})\leq&\nonumber R\sum_{r=-2}^{2}\nu\left\vert \nabla_{r}
f
\right\vert^q+R\sum_{ r=3 }^{\infty} J_0^{r-2}\nu\left\vert \nabla_{\pm r} f
\right\vert^q\\ &\nonumber+\frac{2R}{1-J_0}\sum_{s\in\mathbb{Z}}\nu\left\vert \nabla_{2s}
(\mathbb{E}^{ \{\sim2 s\}}f^q)^\frac{1}{q}
\right\vert^q\\   & \label{3.25Pap2}+\frac{2R}{1-J_0}\sum_{s\in \mathbb{N}}\left(\nu\left\vert \nabla_{2s+2}
(\mathbb{E}^{ \{\sim2 s\}}f^q)^\frac{1}{q}
\right\vert^q+\nu\left\vert \nabla_{2s-2}
(\mathbb{E}^{ \{\sim(-2 s)\}}f^q)^\frac{1}{q}
\right\vert^q\right)\end{align}
  In order to bound the terms involved in the summations in~\eqref{3.25Pap2} we will  apply
Lemma~\ref{lem3.5Pap2} to bound the right hand side of ~\eqref{3.25Pap2}, from which we obtain
\begin{align}\nonumber\nu\mathbb{E}^{ \Gamma_{1}}(f^q\log\frac{f^q}{\mathbb{E}^{ \Gamma_{1}}f^q})\leq & R\sum_{r=-2}^{2}\nu\left\vert \nabla_{r}
f
\right\vert^q+R\sum_{ r=3 }^{\infty} J_0^{r-2}\nu\left\vert \nabla_{\pm r} f
\right\vert^q \\ & \nonumber+\frac{4HR}{1-J_0}\sum_{s\in\mathbb{Z}}\left( \nu\left\vert \nabla_{2s}
f
\right\vert^q+\sum_{ r=2 }^{\infty} J_0^{r-2}\nu\left\vert \nabla_{2s\pm r} f
\right\vert^q\right)\\   \nonumber \leq & R\sum_{r=-2}^{2}\nu\left\vert \nabla_{r}
f
\right\vert^q+\left(\frac{4HR}{1-J^2_0}( 1+\sum_{ r=0 }^{\infty} J_0^{2r} )+R\right)\sum_{s\in\mathbb{Z}}\nu\left\vert \nabla_{2s}
f
\right\vert^q\\  & \nonumber+\left(\frac{4HR}{1-J_0}( 1+\sum_{ r=0 }^{\infty} J_0^{r} )+R\right)\sum_{s\in\mathbb{Z}}\nu\left\vert \nabla_{2s+1}
f
\right\vert^q
\end{align}
where the two  sums are finite since $J_0<1$. The proposition follows for  $$\tilde C=\frac{4HR}{1-J_0}( 1+\sum_{ r=0 }^{\infty} J_0^{r} )+2R$$\qed

\section{Proof of  Proposition~\ref{prp2.4Pap2}.}\label{section5}
In    Proposition~\ref{prp3.4Pap2} and Proposition~\ref{lem1.7P5}  we showed 
a Log-Sobolev type inequality for the one
site measure $\mathbb{E}^{\{\sim i\},\omega}$,  under hypothesis (H1) and (H2) respectively, while    in   Proposition~\ref{prp3.2Pap2} a Spectral Gap type inequality was shown for both cases.  In the following lemma the Spectral Gap type inequality \eqref{2.2Pap2} will  also
be extended
to the product measure $\mathbb{E}^{\Gamma_{i},\omega},i=0,1$. What we will
 show is   that~\eqref{2.1Pap2} for $\mathbb{E}^{\Gamma_{i},\omega},i=0,1$
actually implies~\eqref{2.2Pap2}  for $\mathbb{E}^{\Gamma_{i},\omega},i=0,1$, a basic
result for the usual Log-Sobolev and Spectral Gap inequalities.

\begin{lemma}\label{lem4.3Pap2}Assume  either   (H1) or (H2). If the conditions  (C1)-(C3) are satisfied, then  the following
Spectral Gap type inequality holds  
 $$\nu\left\vert f-\mathbb{E}^{ \Gamma_{i},\omega}f\right\vert^q\ \leq \tilde R\nu\left\vert \nabla
f\right\vert^q$$for $i=0,1$ and  some positive constant $\tilde R$, where $\nabla:=\nabla_{\mathbb{Z}}$.\end{lemma}
\begin{proof}
 To show the lemma we will follow the steps of the proof of the usual LSq implying
the SGq inequality (see \cite{B-Z}).  We will show the inequality for  $i=0$. From Proposition~\ref{prp2.2Pap2} we have
\begin{equation}\label{n4.10Pap2}\nu(\left\vert f\right\vert^qlog\frac{\left\vert f\right\vert^q}{\mathbb{E}^{ \Gamma_0,\omega}\left\vert f\right\vert^q}) \leq \tilde C\nu\left\vert \nabla
f
\right\vert^q\end{equation} Assume without loss of generality that the function $f$ has median zero and denote  $f^+=max(f,0)$ and $f^-=min(f,0)$. Then, according to Lemma 2.2  from \cite{B-Z} and the proof of Theorem 2.1 from the same paper,  we obtain 
$$\mathbb{E}^{ \Gamma_0,\omega}((f^{+})^{q}log\frac{(f^{+})^{q}}{\mathbb{E}^{ \Gamma_0,\omega}(f^{+})^{q}}) \geq\  \log2\mathbb{E}^{ \Gamma_0,\omega}\left((f^{+})^{q}\mathcal{I}_{\{f>0\}}\right)$$
as well as 
$$\mathbb{E}^{ \Gamma_0,\omega}((f^{-})^{q}log\frac{(f^{-})^{q}}{\mathbb{E}^{ \Gamma_0,\omega}(f^{-})^{q}}) \geq\  \log2\mathbb{E}^{ \Gamma_0,\omega}\left((f^{-})^{q}\mathcal{I}_{\{f<0\}}\right)$$
If we apply the Gibbs measure $\nu$ to the last two inequalities we get
\begin{equation}\label{n4.11Pap2}\nu((f^{+})^{q}log\frac{(f^{+})^{q}}{\mathbb{E}^{ \Gamma_0,\omega}(f^{+})^{q}}) \geq\  \log2\nu\left((f^{+})^{q}\mathcal{I}_{\{f>0\}}\right)\end{equation}
and
\begin{equation}\label{n4.12Pap2}\nu((f^{-})^{q}log\frac{(f^{-})^{q}}{\mathbb{E}^{ \Gamma_0,\omega}(f^{-})^{q}}) \geq\  \log2\nu\left((f^{-})^{q}\mathcal{I}_{\{f<0\}}\right)\end{equation}
If we use~\eqref{n4.10Pap2} to bound from above the right hand sides of~\eqref{n4.11Pap2} and~\eqref{n4.12Pap2} we obtain
$$  \tilde C\nu\left(\left\vert \nabla f^{+}
\right\vert^q\mathcal{I}_{\{f>0\}}\right)\geq\  \log2\nu\left((f^{+})^{q}\mathcal{I}_{\{f>0\}}\right)$$
and 
$$  \tilde C\nu\left(\left\vert \nabla f
^{-}\right\vert^q\mathcal{I}_{\{f<0\}}\right)\geq\  \log2\nu\left((f^{-})^{q}\mathcal{I}_{\{f<0\}}\right)$$
If we add the last two and use the estimates  $\left\vert \nabla f
^{+}\right\vert^q\leq\left\vert \nabla f
^{}\right\vert^q$ and $\left\vert \nabla f
^{-}\right\vert^q\leq\left\vert \nabla f
^{}\right\vert^q$ for the gradient, we get$$   \nu(\left\vert f\right\vert^{q})\leq \frac{\tilde C}{\log2}\nu\left\vert \nabla f
\right\vert^q$$
 The last relationship for $f-\mathbb{E}^{ \Gamma _0,\omega}f$ in  place
of $f$ gives \begin{align}\nonumber
\nu\left\vert f-\mathbb{E}^{ \Gamma_0,\omega}f\right\vert^q\leq &\frac{\tilde C}{\log2}\nu\left\vert \nabla_{\Gamma_1}
(f-\mathbb{E}^{ \Gamma_0,\omega}f)
\right\vert^q+\frac{\tilde C}{\log2}\nu\left\vert \nabla_{ \Gamma_0} (f-\mathbb{E}^{ \Gamma_0,\omega}f)
\right\vert^q\\ \leq &\label{4.11Pap2} 2^q\frac{\tilde C}{\log2}\nu\left\vert \nabla_{\Gamma_1}
f\right\vert^q+2^q\frac{\tilde C}{\log2}\nu\left\vert \nabla_{\Gamma_1}
\mathbb{(E}^{ \Gamma_0,\omega}f)
\right\vert^q+\frac{\tilde C}{\log2}\nu\left\vert \nabla_{ \Gamma_0} f\right\vert^q\end{align}
If we use Lemma~\ref{lem4.2Pap2} to bound the second term in the right hand side of~\eqref{4.11Pap2}
we get
\begin{align*}\nu \vert f-\mathbb{E}^{ \Gamma_0,\omega}f\vert^q\leq(2^q\frac{\tilde C}{\log2}D_{1}+2^q\tilde C)\nu\left\vert \nabla_{\Gamma_1}
f\right\vert^q+(2^q\frac{\tilde C}{\log2}D_{2}+\tilde C)\nu\left\vert \nabla_{ \Gamma_0} f\right\vert^q \end{align*}
and the lemma follows for $\tilde R =\max\{\frac{2^q\tilde CD_1+2^q\tilde C}{\log2},\frac{2^q\tilde CD_{2}+\tilde C}{\log2}\}$.
\end{proof}
  Now we can prove
 Proposition~\ref{prp2.4Pap2}.

 ~

\noindent
\textit{\textbf{Proof of  Proposition~\ref{prp2.4Pap2}.}}   In order to  show Proposition \ref{prp2.4Pap2} we can follow  \cite{G-Z} as in \cite{I-P} and \cite{Pa1} for the case of quadratic and non quadratic interactions respectively. In both these two cases    $\mathbb{E}^{i,\omega} $ satisfied an LSq uniformly on $\omega$, that implies that the product measures $\mathbb{E}^{ \Gamma_{i},\omega}$ satisfied the Log-Sobolev q inequality and thus the q Spectral Gap inequality. In the case of Proposition~\ref{prp2.4Pap2} we have assumed the weaker assumptions (H1) and (H2), in which case we can use the weaker result of  Lemma ~\ref{lem4.3Pap2} $$\nu\left\vert f-\mathbb{E}^{ \Gamma_{i},\omega}f\right\vert^q\ \leq \tilde R\nu\left\vert \nabla_{\Gamma_{i}}
f\right\vert^q+ \tilde R\nu\left\vert \nabla_{\Gamma_{j}}
f\right\vert^q$$
for $\{i,j\}=\{0,1\}$.
   For $i\not=j$ we then have that
 \begin{align}\nonumber\nu\vert\mathbb{E}^{\Gamma_{j}} f- \mathbb{E}^{\Gamma_{i}}\mathbb{E}^{\Gamma_{j}} f\vert^q&=\nu\mathbb{E}^{\Gamma_{i}}\vert\mathbb{E}^{\Gamma_{j}} f- \mathbb{E}^{\Gamma_{i}}\mathbb{E}^{\Gamma_{j}} f\vert^q\\ &
  \label{4.12Pap2}\leq \tilde R\nu\left\vert \nabla_{\Gamma_i}(\mathbb{E}^{\Gamma_{j}} f
)\right\vert^q\end{align}          
the last inequality from Lemma~\ref{lem4.3Pap2} for  the measures  $\mathbb{E}^{\Gamma_{0}}$  and  $\mathbb{E}^{\Gamma_{1}}$.
 If we use Lemma~\ref{lem4.2Pap2} to bound the last term in the right hand side of~\eqref{4.12Pap2} we get
$$\nu\vert\mathbb{E}^{\Gamma_{j}} f- \mathbb{E}^{\Gamma_{i}}\mathbb{E}^{\Gamma_{j}} f\vert^q\leq \tilde RD_{1}\nu\left\vert \nabla_{\Gamma_i} f
\right\vert^q+ \tilde RD_{2}\nu\left\vert \nabla_{\Gamma_j} f
\right\vert^q\\ $$  
  From the last inequality we obtain that for any   $n\in \mathbb{N}$, 
 \begin{align*}\nu\vert \mathcal{P}^{n}f- \mathbb{E}^{\Gamma_0}\mathcal{P}^{n} f\vert^q &\leq  \tilde RD_1\nu \vert \nabla_{\Gamma_0}(\mathbb{E}^{\Gamma_0}\mathcal{P}^{n-1} f)\vert^q+\tilde RD_2\nu \vert \nabla_{\Gamma_1}(\mathbb{E}^{\Gamma_0}\mathcal{P}^{n-1} f)\vert^q\\&=\tilde RD_2\nu \vert \nabla_{\Gamma_1}(\mathbb{E}^{\Gamma_0}\mathcal{P}^{n-1} f)\vert^q\end{align*}
 If we use Lemma~\ref{lem4.2Pap2} to bound the last expression we have the following 
\begin{equation} \label{BorelCant1} \nu\vert \mathcal{P}^{n}f- \mathbb{E}^{\Gamma_0}\mathcal{P}^{n} f\vert^q \leq \hat RD_2^{n}\left(D_1\nu\left\vert \nabla_{\Gamma_1} f
\right\vert^q+D_2\nu\left\vert \nabla_{\Gamma_0} f
\right\vert^q\right)\end{equation} Similarly we obtain 
 \begin{equation} \label{BorelCant1+2} \nu\vert\mathbb{E}^{\Gamma_0} \mathcal{P}^{n}f- \mathcal{P}^{n+1} f\vert^q\leq  \hat RD_2^{n}\left(D_1\nu\left\vert \nabla_{\Gamma_1} f
\right\vert^q+D_2\nu\left\vert \nabla_{\Gamma_0} f
\right\vert^q\right)\end{equation}

Consider the sequence $\{\mathcal{Q}^n\}_{n\in \mathbb{N}}$ defined as 
$$\mathcal{Q}^{n}f=\begin{cases}\mathcal{P}^{\frac{n}{2}}f & \text{ if \ } n \text{ \ even} \\
\mathbb{E}^{\Gamma_0}\mathcal{P}^{\frac{n-1}{2}}f & \text{ if \ } n \text{ \ odd} \\
\end{cases}$$
for every $n\in\mathbb{N}$. Hence, if we define the sets $$A_n=\{\vert \mathcal{Q}^nf- \mathcal{Q}^{n+1} f\vert \geq (\frac{1}{2})^n\}$$ we obtain 
\begin{equation*}\nu (A_n)=\nu\left( \{\vert \mathcal{Q}^nf- \mathcal{Q}^{n+1} f\vert \geq (\frac{1}{2})^n\} \right)\leq 2^{qn}\nu\vert \mathcal{Q}^nf- \mathcal{Q}^{n+1} f\vert^q  \end{equation*}by Chebyshev inequality. If we use (\ref{BorelCant1}) and (\ref{BorelCant1+2}) to bound the last we have 
\begin{equation*}\nu (A_n)\leq (2^{q}D_2^\frac{1}{2} )^{n}\hat R\left(D_1\nu\left\vert \nabla_{\Gamma_1} f
\right\vert^q+D_2\nu\left\vert \nabla_{\Gamma_0} f
\right\vert^q\right)
\end{equation*}
We can choose $J$ sufficiently small such that $2^{q}D_2^\frac{1}{2}<\frac{1}{2}$ in which case we get that \begin{equation*}\sum_{n=0}^\infty\nu (A_n)\leq\left(\sum_{n=0}^\infty(\frac{1}{2})^{n} \right)\hat R\left(D_1\nu\left\vert \nabla_{\Gamma_1} f
\right\vert^q+D_2\nu\left\vert \nabla_{\Gamma_0} f
\right\vert^q\right)<\infty
\end{equation*}   From the  Borel-Cantelli lemma (see \cite{Pa1} and \cite{Pa2}), we finally get $$\lim_{n\rightarrow \infty}\mathcal{P}^n f=\lim_{n \text{\ even}, n\rightarrow \infty}\mathcal{Q}^nf=\nu f,  \  \   \nu \  \  a.e. $$
\qed  
 
\section{Conclusion.}
In the present work,  we focus on the $q$ Logarithmic Sobolev Inequality  (LSq) for the infinite dimensional  Gibbs measure related to systems of  spins with values in the Heisenberg group.  

We considered two cases, that of a one site boundary-free measure that satisfies a q Log-Sobolev inequality and that of a one site measure with boundary conditions that satisfies a non uniform U-bound.  In both cases we determined conditions  for the infinite volume Gibbs measure  to satisfy the Log-Sobolev
 Inequality.
 
In this way, the work of \cite{H-Z} was extended to the infinite dimensional setting. In particular we have relaxed the conditions obtained in  \cite{Pa1}  about a similar problem where the case of one site measures $\mathbb{E}^{i,\omega}$ that satisfied an LSq uniformly on the boundary conditions,  where considered.

Furthermore,  the   criterion presented
 in Theorem~\ref{thm2.1Pap2}  can     in particular be applied in the case of local specifications $\{\mathbb{E}^{\Lambda,\omega}\}_{\Lambda\subset\subset \mathbb{Z} ,\omega \in
\Omega}$    with   no quadratic interactions for which $$\left\Vert \nabla_i \nabla_j V(x_i,x_j) \right\Vert_{\infty}=\infty$$
Thus, we have shown that our results  can
go    beyond the usual uniform boundness of the second derivative of the interactions considered in  \cite{Z1}  and \cite{O-R} for real valued variables as well as in \cite{I-P} for spins on the Heisenberg group.

Concerning the additional conditions (C1) placed
here to handle the  interactions, they refer to finite dimensional measures with no boundary conditions which
are easier to handle  than the  $\{\mathbb{E}^{\Lambda,\omega}\}_{\Lambda\subset\subset \mathbb{Z} ,\omega \in
\Omega}$  measures or the
infinite dimensional Gibbs measure $\nu$.

\end{document}